\documentclass[12pt,a4paper,twoside,reqno]{amsart}
\title[On the K\"ahler--Yang--Mills--Higgs  equations]
{On the K\"ahler--Yang--Mills--Higgs equations}

\author[L. \'Alvarez-C\'onsul]{Luis \'Alvarez-C\'onsul}
\address{Instituto de Ciencias Matem\'aticas (CSIC-UAM-UC3M-UCM)\\ Nicol\'as Cabrera 13--15, Cantoblanco\\ 28049 Madrid, Spain}
  \email{l.alvarez-consul@icmat.es}
\author[M. Garcia-Fernandez]{Mario Garcia-Fernandez}
\address{Dep. Matem\'aticas\\ Universidad Aut\'onoma de Madrid\\ and
  Instituto de Ciencias Matem\'aticas (CSIC-UAM-UC3M-UCM)\\ Ciudad
  Universitaria de Cantoblanco\\ 28049 Madrid, Spain}
\email{mario.garcia@icmat.es}
\author[O. Garc\'{\i}a-Prada]{Oscar Garc\'{\i}a-Prada}
\address{Instituto de Ciencias Matem\'aticas (CSIC-UAM-UC3M-UCM)\\ Nicol\'as Cabrera 13--15, Cantoblanco\\ 28049 Madrid, Spain}
  \email{oscar.garcia-prada@icmat.es}

\dedicatory{To Simon Donaldson on his 60th birthday}

\thanks{Partially supported by the Spanish MINECO under ICMAT Severo Ochoa project No. SEV-2015-0554, and under grant No.MTM2016-81048-P}

\usepackage{hyperref,url}
\usepackage{amssymb,latexsym}
\usepackage{amsmath,amsthm}
\usepackage[latin1]{inputenc}
\usepackage{enumerate}
\usepackage[all]{xy} \CompileMatrices
\SelectTips{cm}{12} % computer modern fonts of size 12 for arrowheads

\usepackage{verbatim} %this package has the ``comment'' environment

\usepackage{wrapfig}
\usepackage{graphicx}

\def\YYint#1#2#3{{\setbox0=\hbox{$#1{#2#3}{\int}$}
    \vcenter{\hbox{$#2#3$}}\kern-.52\wd0}}

\theoremstyle{plain}
\newtheorem{theorem}{Theorem}[section]
\newtheorem{lemma}[theorem]{Lemma}
\newtheorem{corollary}[theorem]{Corollary}
\newtheorem{proposition}[theorem]{Proposition}

\newtheorem*{theorem*}{Theorem}
\theoremstyle{definition}
\newtheorem{definition}[theorem]{Definition}
\newtheorem{definition-theorem}[theorem]{Definition-Theorem}

\newtheorem*{acknowledgements}{Acknowledgements}
\theoremstyle{remark}

\numberwithin{equation}{section} 

\newcommand{\tr}{\operatorname{tr}}
\newcommand{\pr}{p}%{\operatorname{pr}}
\newcommand{\Id}{\operatorname{Id}}

\newcommand{\Hom}{\operatorname{Hom}}
\newcommand{\ad}{\operatorname{ad}}
\newcommand{\Ad}{\operatorname{Ad}}
\newcommand{\Aut}{\operatorname{Aut}}

\newcommand{\dbar}{\bar{\partial}}
\newcommand{\imag}{\mathop{{\fam0 {\textbf{i}}}}\nolimits}

\newcommand{\CC}{{\mathbb C}}
\newcommand{\PP}{{\mathbb P}}

\newcommand{\RR}{{\mathbb R}}

\newcommand{\rk}{\operatorname{rk}}

\renewcommand{\(}{\left(}
\renewcommand{\)}{\right)}

\newcommand{\Vol}{\operatorname{Vol}}
\newcommand{\defeq}{\mathrel{\mathop:}=} %{\coloneqq} %{:=}

\newcommand{\surj}{\to\kern-1.8ex\to}

\newcommand{\lto}{\longrightarrow}
\newcommand{\lra}[1]{\stackrel{#1}{\longrightarrow}}

\newcommand{\cA}{\mathcal{A}}

\newcommand{\cJ}{\mathcal{J}}
\newcommand{\cJi}{\mathcal{J}^{i}}
\newcommand{\cK}{\mathcal{K}}

\newcommand{\cF}{\mathcal{F}}

\newcommand{\cG}{\mathcal{G}}

\newcommand{\cO}{\mathcal{O}}

\newcommand{\cR}{\mathcal{R}}
\newcommand{\cS}{\mathcal{S}}

\newcommand{\cT}{{\mathcal{T}}}

\newcommand{\Lie}{\operatorname{Lie}}

\newcommand{\LieG}{\operatorname{Lie} \cG}
\newcommand{\cX}{{\widetilde{\mathcal{G}}}}
\newcommand{\LieX}{\operatorname{Lie} \cX}

\newcommand{\cH}{\mathcal{H}} %\newcommand{\cH}{\operatorname{Ham}(X)}
\newcommand{\LieH}{\Lie\cH}

\newcommand{\GG}{{K^c}} % the symmetry group of a flag manifold
\newcommand{\LieGG}{\mathfrak{k}^c} % the Lie algebra of \GG
\newcommand{\GL}{\operatorname{GL}}

\newcommand{\U}{\operatorname{U}}

\newcommand{\Diff}{\operatorname{Diff}}

\newcommand{\glg}{\mathfrak{g}}

\DeclareMathOperator{\Tr}{Tr}

\begin{document}

\begin{abstract}
In this paper we introduce a set of equations on a principal bundle over a
compact complex manifold coupling a connection on the principal bundle, 
a section of an associated bundle with K\"ahler fibre, and a K\"ahler structure
on the base.  These equations are a generalization of the K\"ahler--Yang--Mills equations introduced by the authors. They also generalize the constant scalar
curvature for a K\"ahler metric studied by Donaldson and others, as well as
the Yang--Mills--Higgs equations studied by Mundet i Riera. We provide a moment
map interpretation of the equations, construct some first examples, and study obstructions to the existence of solutions.
\end{abstract}

\maketitle

\setlength{\parskip}{5pt}
\setlength{\parindent}{0pt}

\tableofcontents

%%%%%%%%%%%%%%%%%%%%%%%%%%
\section{Introduction}
\label{sec:intro}
%%%%%%%%%%%%%%%%%%%%%%%%%%

In the 1990s, Donaldson and Fujiki observed independently that moment
maps play a central role in K\"ahler geometry~\cite{D1,Fj}. Since
then, they have been fruitfully applied in the problem of finding
constant scalar curvature K\"ahler metrics, acting as a guiding
principle for many advances in this topic such as the recent solution
of the K\"ahler--Einstein problem~\cite{ChDoSun}. As noticed
in~\cite{AGG}, the moment map picture for K\"ahler metrics extends to
the study of equations coupling a K\"ahler metric on a compact complex
manifold and a connection on a principal bundle over it, known as the
K\"ahler--Yang--Mills equations. Alike the constant scalar curvature
K\"ahler metrics can be used to understand the moduli space of
polarised manifolds, these equations are natural in the study of the
algebro-geometric moduli problem for bundles and varieties, suggested
in~\cite{Yau2005}. %, and were inspired by the unified field theories in physics.

Motivated by the search of the simplest non-trivial solutions of the
K\"ahler--Yang--Mills equations, the authors studied~\cite{AGG2} the
dimensional reduction of the equations on the product of a Riemann
surface with the complex projective line. This approach to the
K\"ahler--Yang--Mills equations provided a new theory for abelian
vortices on the Riemann surface~\cite{Brad,Noguchi,G1,G3} with
back-reaction of the metric, described by solutions of the
`gravitating vortex equations', and showed an unexpected relation with
the physics of cosmic strings~\cite{AGG2,AGGP}. The further coupling
of a K\"ahler metric and a connection with a `Higgs field' considered
in these works also reveils newly emergent phenomena, not observed in
the theory originally introduced in~\cite{AGG}.

Building on~\cite{AGG,AGG2,AGGP}, this paper develops some basic
pieces of a general moment-map theory for the coupling of a K\"ahler
metric on a compact complex manifold $X$, a connection on a principal
bundle $E$ over $X$, and a Higgs field $\phi$, given by a section of a
K\"ahler fibration associated to $E$. Our treatment of the Higgs field
$\phi$ is inspired by, on the one hand, work on the Yang--Mills--Higgs
equations by Mundet i Riera~\cite{MR}, and, other hand, Donaldson's
study of actions of diffeomorphism groups on spaces of sections of a
bundle~\cite{D7} (see Section~\ref{sec:extended-group}). As we will
see, the K\"ahler--Yang--Mills--Higgs equations introduced in this
paper lead to a very rich theory (see Section~\ref{sec:KYMH}), which
comprises a large class of interesting examples of moment-map
equations (see Sections~\ref{sec:quivers} and~\ref{sec:examples}). In
addition, we expect that these equations may provide a natural
framework for the interaction of K\"ahler geometry and a certain class
of unified field theories in physics~\cite{Yangbook} (see
Section~\ref{sub:gravvortex}).

\begin{acknowledgements}
The authors wish to thank D. Alfaya and T. L. G\'omez for useful discussions.
\end{acknowledgements}

%%%%%%%%%%%%%%%%%%%%%%%%%%%%%%%%%%%%%%%%%%%%%%%%%%%%%%%%%%%%%%%%%%%%%%%%%
\section{Hamiltonian actions of the extended gauge group}
\label{sec:extended-group}
%%%%%%%%%%%%%%%%%%%%%%%%%%%%%%%%%%%%%%%%%%%%%%%%%%%%%%%%%%%%%%%%%%%%%%%%%%

\subsection{The space of connections}
\label{subsec:extended-group}

Details for this section can be found in \cite{AGG}.

Let $X$ be a compact symplectic manifold of dimension $2n$, with
symplectic form $\omega$. Let $G$ be a compact Lie group with Lie algebra 
$\mathfrak{g}$ and $E$ be a smooth
principal $G$-bundle on $X$, with projection map $\pi\colon E\to
X$. Let $\cH$ be the group of Hamiltonian symplectomorphisms of
$(X,\omega)$ and $\Aut E$ be the group of automorphisms of the bundle
$E$. Recall that an \emph{automorphism} of $E$ is a $G$-equivariant
diffeomorphism $g\colon E\to E$. Any such automorphism covers a unique
diffeomorphism $\check{g}\colon X\to X$, i.e. a unique $\check{g}$
such that $\pi\circ g=\check{g}\circ \pi$. We define the
\emph{Hamiltonian extended gauge group} (to which we will simply refer
as extended gauge group) of $E$,
\[
\cX \subset \Aut E,
\]
as the group of automorphisms which cover elements of $\cH$. Then the
gauge group of $E$  is the
normal subgroup $\cG\subset\cX$ of automorphisms covering the
identity.

The map $\cX \lra{\pr} \cH$ assigning to each automorphism $g$ the
Hamiltonian symplectomorphism $\check{g}$ that it covers is
surjective. We thus have an exact sequence of
Lie groups
\begin{equation}
\label{eq:coupling-term-moment-map-1}
  1\to \cG \lra{\iota} \cX \lra{\pr} \cH \to 1,
\end{equation}
where $\iota$ is the inclusion map.

The spaces of smooth $k$-forms on $X$ and smooth $k$-forms with values in
any given vector bundle $F$ on $X$ are denoted by $\Omega^{k}$ and
$\Omega^{k}(F)$, respectively. Fix a positive definite inner product
on $\mathfrak{g}$, invariant under the adjoint action, denoted
\begin{equation}\label{eq:Lie-algebra-inner-product}
  (\cdot,\cdot)\colon \mathfrak{g} \otimes \mathfrak{g} \lto \RR.
\end{equation}
This product induces a metric on the adjoint bundle $\ad E=E\times_G
\mathfrak{g}$, which extends to a bilinear map on $(\ad E)$-valued
differential forms (we use the same notation as in~\cite[\S 3]{AB})
\begin{equation}
\label{eq:Pairing}
\Omega^p(\ad E) \times \Omega^q(\ad E) \lto \Omega^{p+q}
  \colon (a_p,a_q) \longmapsto a_p\wedge a_q.
\end{equation}
We consider the operator
\begin{equation}
\label{eq:Lambda}
  \Lambda=\Lambda_\omega :\Omega^k\lto \Omega^{k-2}\colon \psi \longmapsto \omega^{\sharp} \lrcorner \psi,
\end{equation}
where $\sharp$ is the operator acting on $k$-forms induced by the
symplectic duality $\sharp\colon T^*X \to TX$ and $\lrcorner$ denotes
the contraction operator. Its linear extension to $\Omega^k(\ad E)$ is
also denoted $\Lambda:\Omega^k(\ad E)\to \Omega^{k-2}(\ad E)$ (we use
the same notation as, e.g., in~\cite{D3}).

Let $\cA$ be the set of connections on $E$. This is an affine space
modelled on $\Omega^1(\ad E)$. The 2-form on $\cA$ defined by
\begin{equation}
\label{eq:SymfC}
\omega_{\cA}(a,b) = \int_X a \wedge b \wedge \frac{\omega^{n-1}}{n-1!}
\end{equation}
for $a,b \in T_A \cA = \Omega^1(\ad E)$, $A\in\cA$, is a
symplectic form. 

There is an action of $\Aut E$, and hence of the extended gauge group, on the
space $\cA$ of connections on $E$. To define this action, we view the elements
of $\cA$ as $G$-equivariant splittings $A\colon TE\to VE$ of the short exact
sequence
\begin{equation}
\label{eq:principal-bundle-ses}
  0 \to VE\lto TE\lto \pi^*TX \to 0,
\end{equation}
where $VE=\ker d\pi$ is the vertical bundle. Using the action of $g\in \Aut E$
on $TE$, its action on $\cA$ is given by $g \cdot A \defeq g\circ A \circ
g^{-1}$. Any such splitting $A$ induces a vector space splitting of the Atiyah
short exact sequence
\begin{equation}
\label{eq:Ext-Lie-alg-3}
0\to \LieG \lra{\iota} \Lie(\Aut E) \lra{\pr} \Lie(\Diff X) \to
0
\end{equation}
(cf.~\cite[equation~(3.4)]{AB}), where $\Lie(\Diff X)$ is the Lie algebra of
vector fields on $X$ and $\Lie(\Aut E)$ is the Lie algebra of $G$-invariant
vector fields on $E$. Abusing of the notation, this splitting is given by maps
\begin{equation}
\label{eq:theta-thetaperp}
A\colon \Lie(\Aut E)\lto \LieG, \quad A^\perp\colon
\Lie(\Diff X) \lto \Lie(\Aut E)
\end{equation}
such that $\iota\circ A + A^\perp\circ \pr =\Id$, where $A$
is the vertical projection and $A^\perp$ the horizontal
lift of vector fields on $X$ to vector fields on $E$, given by the connection.

It is easy to see that the $\cX$-action on $\cA$ is symplectic. An equivariant moment map 
for this action was calculated in \cite{AGG}. To give an explicit formula, we use that the splitting~\eqref{eq:theta-thetaperp} restricts to a splitting of the exact
sequence
\begin{equation}
\label{eq:Ext-Lie-alg-2}
 0 \to \LieG \lra{\iota} \LieX \lra{\pr} \LieH \to 0
\end{equation}
induced by~\eqref{eq:coupling-term-moment-map-1}. Consider the
isomorphism of Lie algebras
\begin{equation}
\label{eq:LieH}
\LieH\cong C^{\infty}_0(X),
\end{equation}
where $\LieH$ is the Lie algebra of Hamiltonian vector fields on $X$
and $C^{\infty}_0(X)$ is the Lie algebra of smooth real functions on
$X$ with zero integral over $X$ with respect to $\omega^{n}$, with
the Poisson bracket. This isomorphism is induced by the map
$C^{\infty}(X)\to \LieH \colon f\mapsto \eta_f$, which to each
function $f$ assigns its Hamiltonian vector field $\eta_f$,
defined by
\begin{equation}
\label{eq:eta_phi}
df= \eta_f \lrcorner \omega.
\end{equation}
Let $F_A\in\Omega^2(\ad E)$ be  the curvature of $A\in\cA$ and $z$ be an
element of the space
\begin{equation}
\label{eq:centre-z}
  \mathfrak{z}=\mathfrak{g}^G
\end{equation}
of elements of $\mathfrak{g}$ which are invariant under the adjoint
$G$-action, that we identify with sections of $\ad E$. We have the following.

\begin{proposition}
\label{prop:momentmap-X}
The $\cX$-action on $\cA$ is Hamiltonian, with equivariant moment map
$\mu_\cX\colon \cA\to (\LieX)^*$ given by
\begin{equation}
\label{eq:thm-muX}
  \langle \mu_\cX(A),\zeta\rangle = \int_X A\zeta \wedge (\Lambda F_A - z) \frac{\omega^{n}}{n!} -\frac{1}{4}\int_X f\(\Lambda^2 (F_A\wedge F_A) - 4 \Lambda F_A \wedge z\)\frac{\omega^{n}}{n!}
\end{equation}
for all $\zeta\in\LieX$, $A\in\cA$, where $f\in C_0^\infty(X)$ corresponds to $p(\zeta)$ via \eqref{eq:Ext-Lie-alg-2} and \eqref{eq:LieH}. 
\end{proposition}

%%%%%%%%%%%%%%%%%%%%%%%%%%%%%%%%%%%%%%%%%%%%%%%%%%%%%%%%%%%%%%%%%%%%%%%%
\subsection{Sections of a K\"ahler fibration}
\label{sub:Ham-action}
%%%%%%%%%%%%%%%%%%%%%%%%%%%%%%%%%%%%%%%%%%%%%%%%%%%%%%%%%%%%%%%%%%%%%%%%%

Let $(F,\hat J, \hat \omega)$ be a (possibly non-compact) K\"ahler manifold, with complex structure $\hat J$ and K\"ahler form $\hat \omega$. Following the notation of the previous section, we assume that $G$ acts on $F$ by
Hamiltonian isometries, and fix a $G$-equivariant moment map
$$
\hat \mu \colon F \to \mathfrak{g}^*.
$$
Consider the associated fibre bundle $\cF= E \times_G F$ with fibre $F$. We will denote by $V\cF \subset T\cF$ the vertical bundle of the fibration. 

Let $\cS\defeq\Omega^0(X,\cF)$ the space of $C^\infty$ global sections
of the fibre bundle $\cF$. Using the K\"ahler structure on the fibres
of $\cF$, we endow the infinite-dimensional space $\cS$ with a
K\"ahler structure. Given $\phi \in \cS$, the symplectic form is given
explicitly by
$$
\omega_\cS(\dot \phi_1, \dot \phi_2) = \int_X \hat \omega(\dot \phi_1, \dot \phi_2) \frac{\omega^n}{n!}
$$
where $\dot \phi_i \in T_\phi \cS$ are identified with elements in $\Omega^0(\phi^*V\cF)$. 

An equivariant moment map for the action of the gauge group $\cG$ of $E$ on $(\cS,\omega_\cS)$ was calculated in \cite{MR}. 
Here we are interested in a generalization of this result, where the gauge group is extended by the group of hamiltonian symplectomorphisms $\cH$ of $(X,\omega)$.   The action of the extended  group $\cX$  on $E$ induces an action on $\cS$. This can be seen, for example, by regarding a section of $\cF$ as a $G$-equivariant map $\phi \colon E \to F$. Furthermore, it is easy to see that $\cX$-action on $\cS$ preserves the K\"ahler structure.

To compute the  moment map, let us assume for a moment that the symplectic form $\hat \omega$ is exact (this is, e.g., the situation considered in \cite{AGGP}), that is, there exists $\hat \sigma \in \Omega^1(F)$ such that
$$
d \hat \sigma = \hat \omega.
$$
By averaging over $G$, we can assume that $\hat \sigma$ is invariant under the action of $G$, and it follows that $\omega_\cS = d \sigma_\cS$, with
$$
\sigma_\cS(\dot \phi) = \int_X \hat \sigma(\dot \phi) \frac{\omega^n}{n!}.
$$ 
Then, a $\cX$-equivariant moment map $\mu_\cX \colon\cS\lto(\LieX)^*$ is given by
\begin{equation}\label{eq:mmapsigma}
\langle\mu_\cX,\zeta\rangle = - \sigma_\cS(Y_{\zeta}) = \int_X \hat \sigma(d\phi(\zeta)) \frac{\omega^n}{n!},
\end{equation}
where $Y_\zeta$ denotes the infinitesimal action
$$
Y_{\zeta|\phi} = - d\phi(\zeta)
%\phi^*(\zeta - d\phi (\check \zeta)) \in \Omega^0(\phi^*VY),
$$
of $\zeta\in\LieX$ on $\phi \in \cS$, where $\phi$ is regarded as a map $\phi \colon E \to F$ and we use the identification $E\times_G TF \cong \phi^*V\cF$.
% $\check \zeta$ is vector field on $X$ induced by $\zeta$ (later we will use that $\check \zeta$ is hamiltonian).

We want to obtain an equivalent formula for the moment map \eqref{eq:mmapsigma} which is independent of the choice of $1$-form $\hat \sigma$. For this, choosing a connection $A \colon TE \to VE$ on $E$, we can write
$$
d\phi(\zeta) =  d \phi(A^\perp \zeta) + d\phi (A\zeta) = \check{\zeta} \lrcorner d_A\phi - A \zeta \cdot \phi,
$$
where $\check \zeta := p(\zeta)$, $A \zeta \cdot \phi$ denotes the infinitesimal action of $A \zeta \in \Omega^0(VE)$ along the image of $\phi$ and $d_A \phi = d \phi(A^\perp \cdot) \in \Omega^1(\phi^*V\cF)$ is the covariant derivative induced by $A$. Using that $\hat \sigma$ induces a moment map for the $G$-action on $F$ (that we can assume to be $\hat \mu$) it follows that
$$
\hat \sigma(A \zeta \cdot \phi) = - \langle \phi^*\hat \mu,A\zeta\rangle
$$
where $\phi^*\hat \mu \in \Omega^0(E\times_G \mathfrak{g}^*)$. We use now that $\check \zeta \in \cH$, that is, $\check \zeta \lrcorner \omega = df$ for a smooth function $f \in C^\infty_0(X)$:
\begin{align*}
\int_X \hat \sigma(\check{\zeta} \lrcorner d_A\phi) \frac{\omega^n}{n!} & = \int_X \hat \sigma(d_A\phi)\wedge df \wedge \frac{\omega^{n-1}}{(n-1)!}\\
& = \int_X f d(\hat \sigma(d_A\phi))\wedge \frac{\omega^{n-1}}{(n-1)!}.
\end{align*}
Finally, our desired formula follows from
$$
d(\hat \sigma(d_A\phi)) = \frac{1}{2}\hat \omega(d_A\phi,d_A\phi) + \hat \sigma (F_A \cdot \phi) = \frac{1}{2} \hat \omega(d_A\phi,d_A\phi) - \langle \phi^*\hat \mu, F_A \rangle.
$$
The next result is independent of the existence of the $1$-form $\hat \sigma$ on $F$.

\begin{proposition}
\label{lem:mmapc}
The $\cX$-action on $\cS$ is Hamiltonian, with equivariant moment map
\[
\mu_\cX \colon\cS\lto(\LieX)^*.
\]
For any choice of unitary connection $A$ on $E$, the moment map is given explicitly by
\begin{equation}\label{eq:mmap1}
\begin{split}
\langle\mu(\phi),\zeta\rangle & = \int_X \langle \phi^*\hat \mu,A\zeta\rangle \frac{\omega^n}{n!} + \frac{1}{2} \int_X f (\hat \omega(d_A\phi,d_A\phi) - 2\langle \phi^*\hat \mu, F_A \rangle)\wedge \frac{\omega^{n-1}}{(n-1)!}
\end{split}
\end{equation}
for all $\phi\in\cS$ and $\zeta \in \LieX$ covering $\check \zeta \in \cH$, such that $df = \check \zeta \lrcorner \omega$ %$\eta_f \in \Lie \cH$, 
with $f \in C^\infty_0(X)$.
\end{proposition}

\begin{proof}
The variation of $\langle \phi^*\hat \mu,A\zeta\rangle$ with respect to $\phi$ is
$$
\langle d\hat \mu (\dot \phi),A\zeta\rangle = \hat \omega (d\phi(A\zeta),\dot\phi).
$$
In addition, we have
\begin{align*}
- \hat \omega (d_A\phi(\check \zeta),\dot\phi)\frac{\omega^n}{n} & = - \hat \omega (d_A\phi ,\dot\phi) \wedge df \wedge \omega^{n-1}\\
& = d(f \hat \omega (d_A\phi ,\dot\phi) \wedge \omega^{n-1}) - f d(\hat \omega (d_A\phi ,\dot\phi)) \wedge \omega^{n-1},
\end{align*}
while the variation of $\hat \omega(d_A\phi,d_A\phi) - 2\langle \phi^*\hat \mu, F_A \rangle$ in the second integral is
$$
\hat \omega(d_A\phi,d_A \dot \phi) + \hat \omega(d_A \dot \phi,d_A \phi) - 2\hat \omega (d\phi(F_A),\dot\phi) = - 2d(\hat \omega (d_A \phi,\dot \phi)).
$$
Formula \eqref{eq:mmap1} follows now integrating by parts.
\end{proof}

%%%%%%%%%%%%%%%%%%%%%%%%%%%%%%%%%%%%%%%%%%%%%%%%%%%%%%%%%
\subsection{The Hermitian scalar curvature as a moment map}
\label{section:cscKmmap}
%%%%%%%%%%%%%%%%%%%%%%%%%%%%%%%%%%%%%%%%%%%%%%%%%%%%%%%%%%

%% We now briefly explain the moment map interpretation of the scalar
%% curvature.

Via its projection into the group of Hamiltonian symplectomorphisms $\cH$ (see \eqref{eq:coupling-term-moment-map-1}), the extended gauge group acts on the space $\cJ$ of compatible almost complex structures on the symplectic manifold $(X,\omega)$. As proved by Donaldson~\cite{D1}, the $\cH$-action on $\cJ$ is Hamiltonian, with moment map given by the Hermitian scalar curvature of the almost K\"ahler manifold. The moment map interpretation of the scalar curvature was first given by Quillen in the case of Riemann surfaces and  Fujiki~\cite{Fj} for the Riemannian scalar curvature of K\"ahler manifolds, and generalized independently in~\cite{D1}.

First we recall the notion of Hermitian scalar curvature of an almost
K\"ahler manifold, we follow closely Donaldson's approach. Fix a compact symplectic manifold $X$ of dimension
$2n$, with symplectic form $\omega$. An almost complex structure $J$
on $X$ is called compatible with $\omega$ if the bilinear form
$g_J(\cdot,\cdot) \defeq \omega(\cdot,J\cdot)$ is a Riemannian metric
on $X$. Any almost complex structure $J$ on $X$ which is compatible
with $\omega$ defines a Hermitian metric on $T^*X$ and there is a
unique unitary connection on $T^*X$ whose (0,1) component is the
operator $\dbar_J\colon \Omega^{1,0}_J\to \Omega^{1,1}_J$ induced by
$J$. The real $2$-form $\rho_J$ is defined as $-i$ times the
curvature of the induced connection on the canonical line bundle $K_X
= \Lambda^n_{\CC}T^{\ast}X$, where $i$ is the imaginary unit
$\sqrt{-1}$. The Hermitian scalar curvature $S_J$ is the real function
on $X$ defined by
\begin{equation}
\label{eq:def-S}
  S_J \omega^{n} = 2n\rho_J \wedge \omega^{n-1}.
\end{equation}
The normalization is chosen so that $S_J$ coincides with the
Riemannian scalar curvature when $J$ is integrable. The space $\cJ$ of
almost complex structures $J$ on $X$ which are compatible with
$\omega$ is an infinite dimensional K\"ahler manifold, with complex
structure $\mathbf{J} \colon T_J\cJ \to T_J\cJ$ and K\"ahler form
$\omega_{\cJ}$ given by
\begin{equation}
\label{eq:SympJ}
\mathbf{J}\Phi \defeq J\Phi \text{ and }
\omega_{\cJ} (\Psi,\Phi) \defeq \frac{1}{2n!}\int_{X}\tr(J\Psi \Phi) \omega^{n},
\end{equation}
for $\Phi$, $\Psi \in T_J\cJ$, respectively. Here we identify $T_J\cJ$
with the space of endomorphisms $\Phi\colon TX \to TX$ such that
$\Phi$ is symmetric with respect to the induced metric
$\omega(\cdot,J\cdot)$ and satisfies $\Phi J = - J \Phi$.

The group $\cH$ of Hamiltonian symplectomorphisms $h\colon X\to X$
acts on $\cJ$ by push-forward, i.e. $h \cdot J \defeq h_{\ast} \circ J
\circ h_{\ast}^{-1}$, preserving the K\"ahler form. As proved by
Donaldson~\cite[Proposition~9]{D1}, the $\cH$-action on $\cJ$ is
Hamiltonian with equivariant moment map $\mu_\cH\colon \cJ\to
(\LieH)^*$ given by
\begin{equation}
\label{eq:scmom}
\langle \mu_\cH(J), \eta_f\rangle = -\int_X f S_J \frac{\omega^n}{n!},
\end{equation}
for $f \in C^{\infty}_0(X)$, identified with an element $\eta_f$
in $\LieH$ by~\eqref{eq:LieH} and~\eqref{eq:eta_phi}. 

As a warm up for our discussion in Section \ref{sec:KYMH}, we note that the $\cH$-invariant subspace $\cJi \subset \cJ$ of integrable almost
complex structures is a complex submanifold (away from its
singularities), and therefore inherits a K\"ahler structure.  Over
$\cJi$, the Hermitian scalar curvature $S_J$ is the Riemannian scalar
curvature of the K\"ahler metric determined by $J$ and $\omega$.
Hence the quotient
\begin{equation}
\label{eq:modulicscK}
\mu_{\cH}^{-1} (0)/\cH,
\end{equation}
where $\mu_\cH$ is now the restriction of the moment map to $\cJi$, is
the moduli space of K\"ahler metrics with fixed K\"ahler form
$\omega$ and constant scalar curvature.  Away from singularities, this
moduli space can thus be constructed as a K\"ahler reduction
(see~\cite{Fj} and references therein for details).

%%%%%%%%%%%%%%%%%%%%%%%%%%%%%%%%%%%%%%%%%%%%%%%%%%%%%%%%%%
\section{The K\"ahler--Yang--Mills--Higgs equations}
\label{sec:KYMH}
%%%%%%%%%%%%%%%%%%%%%%%%%%%%%%%%%%%%%%%%%%%%%%%%%%%%%%%%%%

%%%%%%%%%%%%%%%%%%%%%%%%%%%%%%%%%%%%%%%%%%%%%%%%%%%%
\subsection{The equations as a moment map condition}
\label{sub:KYMH}
%%%%%%%%%%%%%%%%%%%%%%%%%%%%%%%%%%%%%%%%%%%%%%%%%%%%

Fix a compact symplectic manifold $X$ of dimension $2n$ with symplectic form
$\omega$, a compact Lie group $G$ and a smooth principal $G$-bundle $E$ on
$X$. We fix an Ad-invariant inner product $( \cdot, \cdot ) \colon\glg\otimes\glg\to\RR$ on the Lie algebra $\glg$ of $G$. Let $\cJ$ be the space of almost complex structures on $X$ compatible with $\omega$ and $\cA$ the space of connections on $E$. Consider the space of triples
\begin{equation}\label{eq:spc-triples}
\cJ\times\cA\times\cS,
\end{equation}
endowed with the symplectic structure
\begin{equation}\label{eq:symplecticT}
\omega_\cJ + 4 \alpha \omega_\cA + 4\beta \omega_\cS,
\end{equation}
(for a choice of non-zero real coupling constants $\alpha, \beta$). Similarly as in \cite[Proposition 2.2]{AGG}, the space \eqref{eq:spc-triples} has a formally integrable almost complex structure, which is compatible with \eqref{eq:symplecticT} when $\alpha >0$ and $\beta >0$, thus inducing a K\"ahler structure in this case.

By Proposition \ref{lem:mmapc} combined with Proposition 
\ref{prop:momentmap-X} and \ref{eq:scmom}, the diagonal
action of $\cX$ on this space is Hamiltonian (here the action of $\cX$ on 
$\cJ$ is given by projecting to $\cH$), with equivariant moment
map $\mu_{\alpha,\beta}\colon\cJ\times\cA\times\cS\to(\LieX)^*$ given
by
\begin{equation}\label{eq:mutriples}
\begin{split}
\langle \mu_{\alpha,\beta}(J,A,\phi),\zeta\rangle & = 4 \int_X ( A\zeta,\alpha \Lambda F_A + \beta \phi^*\hat \mu - z )\frac{\omega^n}{n!}\\
&- \int_X  f(S_J - 2\beta \Lambda\hat \omega(d_A\phi,d_A\phi) + \alpha \Lambda^2 (F_A\wedge F_A) + 4 (\Lambda F_A, \beta \phi^*\hat \mu - \alpha z))\frac{\omega^n}{n!},
\end{split}
\end{equation}
for any choice of central element $z$ in the Lie algebra $\mathfrak{g}$.

Suppose now that $X$ has K\"ahler structures with K\"ahler form $\omega$. This means
that the subspace $\cJi \subset \cJ$ of integrable almost complex structures compatible with $\omega$ is not empty. Define
\begin{equation}\label{eq:spaceT}
\mathcal{T}\subset\cJ\times\cA\times\cS
\end{equation}
by the conditions
$$
J\in \cJi, \qquad A\in \cA^{1,1}_J, \qquad \dbar_{J,A} \phi = 0,
$$
where $\dbar_{J,A} \phi$ denotes the $(0,1)$-part of $d_A \phi$ with respect to $J$ and $\cA_J^{1,1} \subset \cA$ consists of connections $A$ with $F_A\in\Omega^{1,1}_J(\ad E)$, or equivalently satisfying 
$$
F_A^{0,2_J}=0.
$$ 
Here $\Omega_J^{p,q}(\ad E)$
denotes the space of $(\ad E)$-valued smooth $(p,q)$-forms with
respect to $J$ and $F_A^{0,2_J}$ is the projection of $F_A$ into
$\Omega_J^{0,2}(\ad E)$. This space is in bijection with the space of
holomorphic structures on the principal $G^c$-bundle $E^c$ over $(X,J)$ (see~\cite{Si}).

By definition, $\mathcal{T}$ is a complex subspace of \eqref{eq:spc-triples} (away from its singularities) preserved by the $\cX$-action, and hence it inherits a Hamiltonian $\cX$-action.

\begin{proposition}\label{prop:momentmap-inttriples}
  The $\cX$-action on $\mathcal{T}$ is Hamiltonian with
  $\cX$-equivariant moment map $\mu_{\alpha,\beta}\colon\cT\to(\LieX)^*$
  given by
\begin{equation}\label{eq:prop-mutriples}
\begin{split}
\langle \mu_{\alpha,\beta}(J,A,\phi),\zeta\rangle & = 4 \int_X ( A\zeta,\alpha \Lambda F_A + \beta \phi^*\hat \mu - z )\frac{\omega^n}{n!}\\
&- \int_X  f(S_J + \beta \Delta_{g}|\phi^* \hat \mu|^2 + \alpha \Lambda^2 (F_A\wedge F_A) - 4\alpha (\Lambda F_A, z))\frac{\omega^n}{n!},
\end{split}
\end{equation}
for all $(J,A,\phi)\in\cT$ and $\zeta\in\LieX$, where $\Delta_g$ denotes the Laplacian of  $g = \omega(\cdot,J\cdot)$. % covering $\eta_f \in \LieH$, where $f \in C_0^\infty(S)$.
\end{proposition}

\begin{proof}
Since $(J,A,\phi) \in \cT$, we have $\dbar_A \phi = 0$, and hence
\begin{align*}
\Delta_g |\phi^* \hat \mu|^2 & = 2i \Lambda \dbar \partial |\phi^* \hat \mu|^2 = - 2 \Lambda \hat \omega(d_A\phi,d_A\phi) + 4 \langle \phi^*\hat \mu, \Lambda F_A \rangle.
\end{align*}
The statement follows now from \eqref{eq:mutriples}.
\end{proof}

The zeros of the moment map $\mu_{\alpha,\beta}$, restricted to the space of integrable pairs $\cT$, correspond to a coupled system of partial differential equations which is the object of our next definition.

\begin{definition}
We say that a triple  $(J,A,\phi)\in\cT$ satisfies  the \emph{K\"ahler--Yang--Mills--Higgs equations} with coupling constants $\alpha,\beta \in \RR$ if
\begin{equation}\label{eq:KYMH2}
\begin{split}
\alpha \Lambda F_A + \beta \phi^*\hat \mu& = z,\\
S_J + \beta \Delta_{g}|\phi^* \hat \mu|^2 + \alpha \Lambda^2 (F_A\wedge F_A) - 4 \alpha (\Lambda F_A, z) & = c,
\end{split}
\end{equation}
where $S_J$ is the scalar curvature of the metric $g_J =
\omega(\cdot,J\cdot)$ on $X$, $z$ is an element in the center of $\mathfrak{g}$ and $c \in \RR$.
\end{definition}

The constant $c \in \RR$ in \eqref{eq:KYMH2} is explicitly defined by the identity
\begin{equation}
\label{eq:constant-c}
c [\omega]^n = 2\pi n c_1(X)\cup [\omega]^{n-1}+ 2\alpha n (n-1) p_1(E)\cup [\omega]^{n-2} - 4n c(E)\cup [\omega]^{n-1}
\end{equation}
where $p_1(E) \defeq [F_A \wedge F_A] \in H^4(X,\RR)$ and $c(E) \in H^2(X,\RR)$ are the Chern--Weil classes associated to the $G$-invariant symmetric forms $(\cdot,
\cdot)$ and $(\cdot,z)$ on $\mathfrak{g}$ respectively, and so $c$ only depends on $[\omega]$ and the topology of $E$. % (see~\cite[Ch XII, \S 1]{KNII}).

The set of solutions of \eqref{eq:KYMH2} is invariant under the action of $\cX$ and we define the moduli space of solutions as the set of all solutions modulo the
action of $\cX$.  We can identify this moduli space with the quotient
\begin{equation}
\label{eq:symplecticreduccX}
\mu_{\alpha,\beta}^{-1}(0)/\cX,
\end{equation}
where $\mu_{\alpha,\beta}$ denotes now the restriction of the moment map to
$\cT$.  Away from singularities, this is a K\"ahler quotient for the
action of $\cX$ on the smooth part of $\cT$ equiped with the K\"ahler form obtained by the restriction of \eqref{eq:symplecticT}.

\subsection{Futaki invariant and geodesic stability}
\label{sub:obstructions}

In this section, we explain briefly some general obstructions to the existence of solutions of the K\"ahler--Yang--Mills--Higgs equations \eqref{eq:KYMH2}, which follow the  general method developed in \cite[\S 3]{AGG}. To describe them, it is helpful to adopt a dual view point, based on complex differential geometry.

We fix a compact complex manifold $X$ of dimension $n$, a K\"ahler class $\Omega \in
H^{1,1}(X)$ and a holomorphic principal bundle $E^c$ over $X$. We assume that the structure group of $E^c$ is a complex reductive Lie group $G^c$, and that the Lie algebra $\mathfrak{g}^c$ of $G^c$ is endowed with an $\Ad$-invariant symmetric bilinear form. Let $(F,\hat J, \hat \omega)$ be a (possibly non-compact) K\"ahler manifold, with complex structure $\hat J$ and K\"ahler form $\hat \omega$. We assume that a maximal compact subgroup $G \subset G^c$ acts on $F$ by Hamiltonian isometries, and fix a $G$-equivariant moment map
$$
\hat \mu \colon F \to \mathfrak{g}^*.
$$
Consider the associated fibre bundle $\cF= E^c \times_{G^c} F$ with fibre $F$, and assume that there exists a holomorphic section
$$
\phi \in H^0(X,\cF).
$$
Then, the K\"ahler--Yang--Mills--Higgs equations on $(X,E^c,\phi)$, for fixed coupling constants $\alpha,\beta \in\RR$, are
\begin{equation}\label{eq:KYMH3}
% \left. \begin{array}{l}
\begin{split}
\alpha \Lambda_\omega F_H + \beta \phi^*\hat \mu & =z,\\
S_\omega + \beta \Delta_{\omega}|\phi^* \hat \mu|^2 + \alpha \Lambda_\omega^2 (F_H\wedge F_H) - 4 \alpha (\Lambda_\omega F_H, z) & =c,
\end{split}
\end{equation}
where the unknowns are a K\"ahler metric on $X$ with K\"ahler form
$\omega$ in $\Omega$, and a reduction $H \colon X \to E^c/G$ to $G$. In this case, $F_H$ is the curvature of the Chern connection $A_H$ of $H$ on $E^c$, and $S_\omega$ is the scalar curvature of the K\"ahler metric. Note
that 
%the operator in~\eqref{eq:Lambda} depends on $\omega$, and 
the constant $c\in\RR$ depends on $\alpha$, $\Omega$ and the topology of $X$ and $E^c$. In the rest of this section, we will assume $\alpha > 0$ and $\beta > 0$ in the definition of \eqref{eq:KYMH3}.

Our first obstruction builds on the general method in \cite[\S 3]{AGG} and classical work of Futaki~\cite{Futaki1}. Consider the complex Lie group $\Aut(X,E^c)$ of automorphisms of $(X,E^c)$ and the complex Lie subgroup fixing the section $\phi$
$$
\Aut(X,E^c,\phi) \subset \Aut(X,E^c).
$$
We define a map
\[
\cF_{\alpha,\beta}\colon\Lie\Aut(X,E^c,\phi)\lto\CC
\]
given by the formula
\begin{equation}\label{eq:alphafutakismooth}
\begin{split}
\langle\cF_{\alpha,\beta},\zeta\rangle & = 4 \int_X (A_H\zeta,\alpha \Lambda_\omega F_H + \beta \phi^*\hat \mu - \alpha z )\frac{\omega^n}{n!}\\
&- \int_X  \varphi(S_\omega + \beta \Delta_{\omega}|\phi^* \hat \mu|^2 + \alpha \Lambda_\omega^2 (F_H\wedge F_H) - 4 (\Lambda_\omega F_H, z)))\frac{\omega^n}{n!},
\end{split}
\end{equation}
for a choice a K\"ahler form $\omega \in \Omega$ and hermitian metric
$H$ on $E$. To explain this formula, we note that $\Lie\Aut(X,E^c)$ is
the space of $G^c$-invariant holomorphic vector fields $\zeta$
on the total space of $E^c$. Any such $\zeta$ covers a real-holomorphic vector field $\check{\zeta}$ on $X$, and decomposes, in terms of the connection $A_H$, as
\[
\zeta=A_H\zeta+A_H^\perp\check{\zeta},
\]
where $A_H \zeta$ and $A_H^\perp \check \zeta$ are its vertical and
horizontal parts. The complex-valued function 
\[
\varphi\defeq\varphi_1+i\varphi_2,
\]
with $\varphi_1,\varphi_2\in C^\infty_0(X,\omega)$, is determined by the unique decomposition
\[
\check{\zeta}=\eta_{\varphi_1}+J\eta_{\varphi_2}+\gamma,
\]
valid precisely because $\check{\zeta}$ is a real-holomorphic vector
field,
% $\check{\zeta}$, 
where $J$ is the (integrable) almost complex structure of $X$, $\eta_{\varphi_j}$ (for $j=1,2$) is the Hamiltonian vector field of $\varphi_j$, and $\gamma$ is the dual of a
$1$-form that is harmonic with respect to the K\"ahler metric.

This \emph{Futaki character} provides the following obstruction to the existence of
solutions of the K\"ahler--Yang--Mills--Higgs equations equations
(cf.~\cite[Theorem~3.9]{AGG}). Let $B$ be the space of pairs
$(\omega,H)$ consisting of a K\"ahler form $\omega$ in the cohomology
class $\Omega$ and a reduction $H$ of $E^c$ to $G \subset G^c$.

\begin{proposition}\label{prop:futakibis}
The map \eqref{eq:alphafutakismooth} is independent of the choice of element $(\omega,H)$ in $B$. It defines a character of $\Lie \Aut(X,E^c,\phi)$, which vanishes identically if there exists a solution of the K\"ahler--Yang--Mills--Higgs equations \eqref{eq:KYMH3} with K\"ahler class $\Omega$.
\end{proposition}

Further obstructions to the existence of solutions of the
K\"ahler--Yang--Mills equations are intimately related to the geometry
of %a structure of symmetric space on
the infinite-dimensional space $B$. It is interesting to notice that this geometry is independent of the choice of holomorphic section $\phi$ on $\cF$. The space $B$ has a structure of
symmetric space~\cite[Theorem~3.6]{AGG}, that is, it has a
torsion-free affine connection $\nabla$, with holonomy group contained
in the extended gauge group (each point of $B$ determines one such
group) and covariantly constant
curvature. %The notion of geodesic on $B$, is the standard one, with respect to the connection $\nabla$.
The partial differential equations that define the geodesics
$(\omega_t,H_t)$ on $B$, with respect to the connection $\nabla$, are
\begin{equation}
\label{eq:geodesicequation}
%\left. \begin{array}{l}
\begin{split}
dd^c(\ddot \varphi_t - (d\dot\varphi_t,d\dot\varphi_t)_{\omega_t}) &= 0,\\
\ddot H_t - 2J\eta_{\dot \varphi_t} \lrcorner d_{H_t}\dot H_t + iF_{H_t}(\eta_{\dot \varphi_t},J \eta_{\dot\varphi_t}) &= 0,
\end{split}
%\end{array}\right \}.
\end{equation}
where $\eta_{\dot \varphi_t}$ is the Hamiltonian vector field of $\dot
\varphi_t$ with respect to $\omega_t$,
i.e. $d\dot{\varphi}_t=\eta_{\dot \varphi_t} \lrcorner \omega_t$. Assuming existence
of smooth geodesic rays, that is, smooth solutions $(\omega_t,H_t)$
of~\eqref{eq:geodesicequation} defined on an infinite interval $0\leq
t<\infty$, with prescribed boundary condition at $t=0$, one can define
a stability condition for $(X,E^c,\phi)$. Define a $1$-form $\sigma_{\alpha,\beta}$
on $B$ by
\begin{align*}
\notag
\sigma_{\alpha,\beta}(\dot \omega,\dot H) =
& -4i \int_X (H^{-1}\dot{H},\alpha \Lambda_\omega F_H + \beta \phi^*\hat \mu - \alpha z )\frac{\omega^n}{n!}\\
&- \int_X  \dot \varphi(S_\omega + \beta \Delta_{\omega}|\phi^* \hat \mu|^2 + \alpha \Lambda_\omega^2 (F_H\wedge F_H) - 4 (\Lambda_\omega F_H, z)))\frac{\omega^n}{n!},
\end{align*}
where $(\dot \omega, \dot H)$ is a tangent vector to $B$ at
$(\omega,H)$ and $\dot{\omega}=dd^c\dot{\varphi}$ for $\varphi\in
C_{0}^\infty(X,\omega)$.

\begin{definition}\label{def:geodstab}
The triple $(X,E^c,\phi)$ is \emph{geodesically semi-stable} if for every smooth geodesic ray $b_t$ on $B$, the following holds
$$
\lim_{t \to + \infty}\sigma_{\alpha,\beta}(\dot b_t) \geq 0.
$$
\end{definition}

Under the assumption that $B$ is geodesically convex, that is, that
any two points in $B$ can be joined by a smooth geodesic segment,
\emph{geodesic semi-stability} provides an obstruction to the
existence of solutions of \eqref{eq:KYMH3}. 

The proof of the next proposition follows from the fact that the quantity $\sigma_{\alpha,\beta}(\dot b_t)$ is increasing along
geodesics in $B$, with speed controlled by the infinitesimal action on
the space $\cT$ in~\ref{sec:KYMH} (see the proof of \cite[Proposition
3.14]{AGG}).

\begin{proposition}
Assume that $B$ is geodesically convex. If there exists a solution of the K\"ahler--Yang--Mills--Higgs equations in $B$, then $(X,P^c,\phi)$ is geodesically semi-stable. Furthermore, such a solution is unique modulo the action of $\Aut(X,E^c,\phi)$.
\end{proposition}

The space $B$ defines a geodesic submersion over the symmetric space
of K\"ahler metrics on the class $\Omega$ \cite{D6,Mab1,Se}. In
particular, this implies that in general one cannot expect existence
of smooth geodesic segments on $B$ with arbitrary boundary conditions.

\subsection{Matsushima--Lichnerowicz for the K\"ahler--Yang--Mills--Higgs equations}

In this section we introduce a new obstruction to the existence of solutions of the K\"ahler--Yang--Mills--Higgs equations. This is based on an analogue of Matsushima--Lichnerowicz Theorem \cite{Lichnerowicz,Matsushima} for \eqref{eq:KYMH3}, which relates the existence of a solution on $(X,E^c,\phi)$ with the reductivity of $\Lie \Aut(X,E^c,\phi)$. Our proof relies on the moment-map interpretation of the equations \eqref{eq:KYMH3}, following closely Donaldson--Wang's abstract proof~\cite{D1,W} of the Matsushima--Lichnerowicz Theorem.

For simplicity, we will assume that $X$ has vanishing first Betti number, even though we expect that our analysis goes through with minor modifications to the general case.

\begin{theorem}\label{th:Matsushima-type}
Assume $H^1(X,\RR) = 0$. If $(X,E^c,\phi)$ admits a solution of the K\"ahler--Yang--Mills--Higgs equations \eqref{eq:KYMH3} with $\alpha > 0$ and $\beta > 0$, then the Lie algebra of $\Aut(X,E^c,\phi)$ is reductive.
\end{theorem}

To prove our theorem we need some preliminary results. Let $\omega$ be a K\"ahler form on $X$ and a reduction $H$ of $E^c$ to $G \subset G^c$. The following lemma gives a convenient formula for the elements of $\Lie\Aut(X,E^c,\phi)$ %(and hence for elements of $\mathfrak{g}$),
adapted to the pair $(\omega,H)$, and is reminiscent of the
Hodge-theoretic description of holomorphic vector fields on compact
K\"ahler manifolds (see, e.g.,~\cite[Ch. 2]{Gauduchon}). As in~\eqref{sec:extended-group}, $\LieX$ will denote the Lie algebra of the extended gauge group associated to the symplectic structure $\omega$ and the reduction $E_H$. For the proof, we will not assume that $(\omega,H)$ is a solution of \eqref{eq:KYMH3}. We denote by $I$ the almost complex structure on the total space of $E^c$.

\begin{lemma}\label{lem:Lich-type}
Assume $H^1(X,\RR) = 0$. Then, for any $y \in\Lie\Aut(X,E^c)$ there exist $\zeta_1,\zeta_2\in\LieX$ such that
\begin{equation}\label{eq:ysplit}
y=\zeta_1+I\zeta_2.
\end{equation}
\end{lemma}

\begin{proof}
Let $A$ be the Chern connection of $H$ on $E^c$. We will use the
decomposition of 
\begin{equation}\label{eq:holvectfield-Vert+Horiz}
y = Ay+A^\perp\check{y}  
\end{equation}
into its vertical and horizontal components $Ay$, $A^\perp\check{y}$,
where $\check{y}$ is the unique holomorphic vector field on $X$ covered by $y$. %image of $y$ under~\eqref{eq:infinitesimal-cover-map}. 
Using the anti-holomorphic involution on the Lie algebra $\mathfrak{g}^c$ determined by $G \subset G^c$, we decompose
$$
Ay = \xi_1 + i \xi_2,
$$
for $\xi_j \in \Omega^0(\ad E_H)$. Furthermore, as $H^1(X,\RR) = 0$, we have
\[
\check{y}=\check{y}_1+J\check{y}_2,
\]
where $\check{y}_1$ and $\check{y}_2$ are Hamiltonian vector fields for the symplectic form $\omega$. Hence, defining the vector fields
\[
\zeta_j = \xi_j + A^\perp \check{y}_j, 
\]
for $j=1,2$, we obtain the result.
\end{proof}

We will now apply Lemma~\ref{lem:Lich-type} to the
elements of $\Lie\Aut(X,E^c,\phi)\subset\Lie\Aut(X,E^c)$.

\begin{lemma}\label{lem:y*}
Assume $H^1(X,\RR) = 0$ and that $(X,E^c,\phi)$ admits a solution $(\omega,h)$ of the K\"ahler--Yang--Mills--Higgs equations with $\alpha > 0$ and $\beta > 0$. Then, for any $y \in \Lie\Aut(X,E^c,\phi)$, the vector fields $\zeta_1,\zeta_2$ in \eqref{eq:ysplit} satisfy $\zeta_1,\zeta_2 \in \Lie\Aut(X,E^c,\phi)$.
\end{lemma}

\begin{proof}
By the results of Section \ref{sub:KYMH}, if $(\omega,h)$ is a solution of
 \eqref{eq:KYMH3}, then the triple $t\defeq(J,A,\phi)$ is a zero of a moment map
\[
\mu_{\alpha,\beta} \colon \cT \to \LieX^*
\]
for the action of $\cX$ on the space of `integrable triples' $\cT$ defined in \eqref{eq:spaceT}. Recall that $\cT$ is endowed with a
(formally) integrable almost complex structure $\mathbf{I}$, and
K\"ahler metric
\[
g_{\alpha,\beta} = \omega_{\alpha,\beta}(\cdot,\mathbf{I}\cdot)
\]
(as we are assuming $\alpha>0$ and $\beta >0$), where the compatible symplectic structure
$\omega_{\alpha,\beta}$ is as in \eqref{eq:symplecticT}. Given $y \in \Lie\Aut(X,E^c)$, we denote by $Y_{y|t}$ the infinitesimal action of $y$ on
$t)$. Then the proof reduces to show that
$Y_{\zeta_1|t}=Y_{\zeta_2|t}=0$ for $y \in \Lie\Aut(X,E^c,\phi)$, where $\zeta_1,\zeta_2$ as in \eqref{eq:ysplit}. To prove this, we note that
since the almost complex structure $I$ on $E^c$ is integrable, we have (see~\cite[Section 3.2]{AGG})
\[
0 = Y_{y|t} = Y_{\zeta_1 + I \zeta_2|t} = Y_{\zeta_1|t} + \mathbf{I}Y_{\zeta_2|t}.
\]
Considering now the norm $\|\cdot\|$ on $T_t\cT$ induced by the
metric $g_{\alpha,\beta}$, we obtain
\[
0=\|Y_{y|t}\|^2 = \|Y_{\zeta_1|t}\|^2 + \|Y_{\zeta_2|t}\|^2 - 2 \omega_{\alpha,\beta}(Y_{\zeta_1|t},Y_{\zeta_1|t}).
\]
Now, $\mu_{\alpha,\beta}(t) = 0$ and the moment map $\mu_\alpha$ is
equivariant, so
\[
\omega_\alpha(Y_{\zeta_1|t},Y_{\zeta_1|t})
=d\langle\mu_\alpha,\zeta_1\rangle(Y_{\zeta_2|t})
=\langle\mu_\alpha(t),[\zeta_1,\zeta_2]\rangle=0,
\]
and therefore
\[
\|Y_{\zeta_1|t}\|^2 = \|Y_{\zeta_2|t}\|^2=0,
\]
so we conclude that $\zeta_1,\zeta_2\in \Lie\Aut(X,E^c,\phi)$, as required.
\end{proof}

Theorem \ref{th:Matsushima-type} is now a formal consequence of
Lemma~\ref{lem:y*}.

\begin{proof}[Proof of Theorem~\ref{th:Matsushima-type}]
%The proof is now a formality. 
Considering the $\cX$-action on $\cT$, we note that the Lie algebra
$\mathfrak{k}=\LieX_t$ of the isotropy group $\cX_t$ of the triple
$t=(J,A,\phi)\in\cT$ satisfies
\[
\mathfrak{k} \oplus I \mathfrak{k} \subset \Lie\Aut(X,E^c,\phi).
\] 
Furthermore, the Lie group $\cX_t$ is compact, because it can be
regarded as a closed subgroup of the isometry group of a Riemannian
metric on the total space of $E_H$ (see~\cite[Section 2.3]{AGG}). Now, Lemma~\ref{lem:y*} implies that
\[
\Lie\Aut(X,E^c,\phi)=\mathfrak{k}\oplus I\mathfrak{k},
\]
so $\Lie\Aut(X,E^c,\phi)$ is the complexification of the Lie algebra
$\mathfrak{k}$ of a compact Lie group, and hence a reductive complex Lie algebra.
\end{proof}

\section{Gravitating vortices and dimensional reduction}
\label{sec:quivers}

%Added: \usepackage{bm} % For bold math symbols: \bm{}
\newcommand{\glk}{\mathfrak{k}}
\newcommand{\gll}{\mathfrak{l}}
\newcommand{\glu}{\mathfrak{u}}
\newcommand{\glr}{\mathfrak{r}}
\newcommand{\gls}{\mathfrak{s}}
\newcommand{\glt}{\mathfrak{t}}
\newcommand{\balpha}{{\bm{\alpha}}}
\newcommand{\bepsilon}{{\bm{\varepsilon}}}

% We will now consider in more detail the moment-map interpretation of
% the K\"ahler--Yang--Mills--Higgs equations in the particular case in
% which the Higgs field $\phi$ is a section of a special type of vector
% bundles, defining a quiver bundle
% (Section~\ref{sub:quivers.Equations}) and, in some special cases, give
% another interpretation as the dimensional reduction of the
% K\"ahler--Yang--Mills equations
% (Section~\ref{sub:quivers.DimRed}). Combining both interpretations, we
% obtain new existence results, in the weak limit, for the
% K\"ahler--Yang--Mills equations, on manifolds for which previous
% results cannot be directly applied
% (Section~\ref{sub:quivers.Existence-KYM}).

\subsection{Gravitating quiver vortex equations}
\label{sub:quivers.Equations}

Here we consider in more detail the K\"ahler--Yang--Mills--Higgs
equations when the Higgs field is a section of a special type of
vector bundles, defining a quiver bundle.
To fix notation, we recall the notions of quiver and quiver bundle
(see, e.g.,~\cite{AG1} for details). A \emph{quiver} $Q$ is a pair of
sets $(Q_0,Q_1)$, together with two maps $t,h\colon Q_1\to Q_0$. The
elements of $Q_0$ and $Q_1$ are called the vertices and arrows of the
quiver, respectively. An arrow $a\in Q_1$ is represented pictorially
as $a\colon i\to j$, where $i=ta$ and $j=ha$ are called the tail and
the head of $a$. Suppose for simplicity that the quiver is finite,
that is, both $Q_0$ and $Q_1$ are finite sets (this condition will be
weakened in Section~\ref{sub:quivers.DimRed}). Fix a compact complex
manifold $X$ of dimension $n$. A \emph{holomorphic $Q$-bundle over
  $X$} is a pair $(E,\phi)$ consisting of a set $E$ of holomorphic
vector bundles $E_i$ on $X$, indexed by the vertices $i\in Q_0$, and a
set $\phi$ of holomorphic vector-bundle homomorphisms $\phi_a\colon
E_{ta}\to E_{ha}$, indexed by the arrows $a\in Q_1$. Note that it is
often useful to consider a category of \emph{twisted} quiver bundles
(see~\cite{AG1}), but they will not be needed for the application
given in Corollary~\ref{thm:quivers.Existence-KYM}. %(the moment-map interpretation of their gravitating vortex equations requires extra data to define an action of the extended gauge group on the twisting bundles).

A \emph{Hermitian metric} on $(E,\phi)$ is a set $H$ of %($C^\infty$)
Hermitian metrics $H_i$ on $E_i$, indexed by the vertices $i\in
Q_0$. Any such Hermitian metric determines a $C^\infty$ adjoint
vector-bundle morphism $\phi_a^{*H_a}\colon E_{ha}\to E_{ta}$ of
$\phi_a\colon E_{ta}\to E_{ha}$ with respect to the Hermitian metrics
$H_{ta}$ and $H_{ha}$, for each $a\in Q_1$, and we can construct a
($H$-self-adjoint) `commutator'
\[
[\phi,\phi^{*H}]=\bigoplus_{i\in Q_0}[\phi,\phi^{*H}]_i
\colon\bigoplus_{i\in Q_0} E_i\lto\bigoplus_{i\in Q_0} E_i,
\]
with components 
\[
[\phi,\phi^{*H}]_i
:=\sum_{a\in h^{-1}(i)}\phi_a\circ\phi_a^{*H_a}-\sum_{a\in t^{-1}(i)}\phi_a^{*H_a}\circ\phi_a
\colon E_i\lto E_i,
\]
for all $i\in Q_0$. In the following, $\RR_{>0}\subset \RR$ is the set
of positive real numbers, and for any two sets $I$ and $S$, $S^I$ is
the set of maps $\sigma\colon I\to S,\, i\mapsto \sigma_i$; to avoid
confusion with the symbols used to denote quiver vertices,
$\imag=\sqrt{-1}$ is the imaginary unit.

\begin{definition}
\label{def:quivers.Equations}
Fix constants $\rho\in\RR_{>0}$, $\sigma\in\RR_{>0}^{Q_0}$ and $\tau\in\RR^{Q_0}$.
The \emph{gravitating quiver $(\rho,\sigma,\tau)$-vortex
  equations} for a pair $(\omega,H)$, consisting of a K\"ahler metric
$\omega$ on the complex manifold $X$ and a Hermitian metric $H$ on a holomorphic $Q$-bundle $(E,\phi)$, are
\begin{subequations}\label{eq:quivers.Equations.1}
\begin{align}\label{eq:quivers.Equations.1a}
\sigma_i\imag\Lambda_\omega F_{H_i} + [\phi,\phi^{*H}]_i
%\sum_{a\in h^{-1}(i)}\phi_a\circ\phi_a^{*H_a}-\sum_{a\in t^{-1}(i)}\phi_a^{*H_a}\circ\phi_a
=\tau_i\Id_{E_i},&\\
\label{eq:quivers.Equations.1b}
S_\omega 
-\rho\sum_{i\in Q_0}\sigma_i\Lambda_\omega^2\Tr F_{H_i}^2
+2\rho\sum_{a\in Q_1}\left(\Delta_\omega+2\(\frac{\tau_{ha}}{\sigma_{ha}}-\frac{\tau_{ta}}{\sigma_{ta}}\)\)
\lvert\phi_a\rvert_{H_a}^2
=c.&
\end{align}
\end{subequations}
\end{definition}

Here, 
% $F_{H_i}$ is the curvature 2-form of the Chern connection of the
% Hermitian metric $H_i$ on $E_i$, so $\Tr F_{H_i}^2$ is a 4-form,
% $\Delta_\omega=2\imag\Lambda_\omega\dbar\partial\colon C^\infty(X)\to
% C^\infty(X)$ is the Laplacian acting on functions,
$\lvert\phi_a\rvert_{H_a}^2\defeq\Tr(\phi_a\circ\phi_a^{*H_a})\in
C^\infty(X)$ is the pointwise squared norm, and $c$ is a constant,
determined by the parameters $\rho,\sigma,\tau$, the cohomology class
of $\omega$, and the characteristic classes of the manifold $X$ and
the vector bundles $E_i$. More precisely,
\begin{align*}
c\operatorname{Vol}_\omega(X)
&=2\int_X\rho_\omega\wedge\frac{\omega^{n-1}}{(n-1)!}
-4\rho\sum_{i\in Q_0}\sigma_i\int_X\Tr F_{H_i}^2\wedge\frac{\omega^{n-2}}{(n-2)!}
\\&
+4\rho\operatorname{Vol}_\omega(X)\sum_{i\in Q_0}\(\frac{\tau_i}{\sigma_i}-\mu_\omega(E_i)\)\tau_ir_i,
\end{align*}
where $\Vol_\omega(X)=\int_X\omega^n/n!$, $r_i$ is the rank of $E_i$, its normalized
$\omega$-slope is
\begin{equation}\label{eq:quivers.Slope}
\mu_\omega(E_i)\defeq\frac{1}{\operatorname{Vol}_\omega(X)}\frac{1}{r_i}\int_X\Tr(\imag
F_{A_i})\wedge\frac{\omega^{n-1}}{(n-1)!}, 
\end{equation}
and $\rho_\omega$ is the Ricci form. To see this, we
integrate~\eqref{eq:quivers.Equations.1b}, use~\eqref{eq:def-S}, and
also integrate the following identity (that follows
from~\eqref{eq:quivers.Equations.1a})
\begin{equation}\label{eq:quivers.Equations.constant}
\begin{split}
\sum_{a\in Q_1}\(\frac{\tau_{ha}}{\sigma_{ha}}-\frac{\tau_{ta}}{\sigma_{ta}}\)\lvert\phi_a\rvert_{H_a}^2
&%=\sum_{i\in Q_0}\frac{\tau_i}{\sigma_i}\(\sum_{a\in h^{-1}(i)}\lvert\phi_a\rvert_{H_a}^2-\sum_{a\in t^{-1}(i)}\lvert\phi_a\rvert_{H_a}^2\)
%\\&
=\sum_{i\in Q_0}\frac{\tau_i}{\sigma_i}\Tr[\phi,\phi^{*H}]_i
=\sum_{i\in Q_0}\(\frac{\tau_i^2r_i}{\sigma_i}-\tau_i\Tr(\imag\Lambda_\omega F_{H_i})\).
\end{split}
\end{equation}
% To see this, we integrate~\eqref{eq:quivers.Equations.1b}, and
% use~\eqref{eq:def-S}, obtaining
% \begin{align*}
% c\operatorname{Vol}_\omega(X)
% &=\int_X 2\rho_\omega\wedge\frac{\omega^{n-1}}{(n-1)!}
% -4\rho\sum_{i\in Q_0}\sigma_i\int_X\Tr F_{H_i}^2\wedge\frac{\omega^{n-2}}{(n-2)!}
% \\&
% +4\rho\sum_{a\in Q_1}
% \(\frac{\tau_{ha}}{\sigma_{ha}}-\frac{\tau_{ta}}{\sigma_{ta}}\)\int_X\lvert\phi_a\rvert_{H_a}^2\frac{\omega^n}{n!},
% \end{align*}
% and calculate the last term integrating over $X$ the identity (where
% we have used~\eqref{eq:quivers.Equations.1a})
% \begin{equation}\label{eq:quivers.Equations.constant}
% \begin{split}
% \sum_{a\in Q_1}\(\frac{\tau_{ha}}{\sigma_{ha}}-\frac{\tau_{ta}}{\sigma_{ta}}\)\lvert\phi_a\rvert_{H_a}^2
% &%=\sum_{i\in Q_0}\frac{\tau_i}{\sigma_i}\(\sum_{a\in h^{-1}(i)}\lvert\phi_a\rvert_{H_a}^2-\sum_{a\in t^{-1}(i)}\lvert\phi_a\rvert_{H_a}^2\)
% %\\&
% =\sum_{i\in Q_0}\frac{\tau_i}{\sigma_i}\Tr[\phi,\phi^{*H}]_i
% =\sum_{i\in Q_0}\(\frac{\tau_i^2r_i}{\sigma_i}-\tau_i\Tr(\imag\Lambda_\omega F_{H_i})\).
% \end{split}
% \end{equation}
% obtaining
% \[
% \sum_{a\in Q_1}
% \(\frac{\tau_{ha}}{\sigma_{ha}}-\frac{\tau_{ta}}{\sigma_{ta}}\)\int_X\lvert\phi_a\rvert_{H_a}^2\frac{\omega^n}{n!}
% =\operatorname{Vol}_\omega(X)\sum_{i\in Q_0}
% \(\frac{\tau_i}{\sigma_i}-\mu_\omega(E_i)\)\tau_ir_i.
% \]

Given a fixed K\"ahler form $\omega$ on $X$, the first set of
equations~\eqref{eq:quivers.Equations.1a}, involving %an unknown
a Hermitian metric $H$ on $(E,\phi)$, were called the
\emph{$(\sigma,\tau)$-vortex equations} on $(E,\phi)$ over the
K\"ahler manifold $(X,\omega)$ in~\cite{AG2}, where their symplectic
interpretation and their relation with a
\emph{$(\sigma,\tau)$-polystability condition} were provided.
To explain how the larger set of
equations~\eqref{eq:quivers.Equations.1} fit in the general moment-map
picture of Section~\ref{sec:KYMH}, we now fix the metrics and consider
the holomorphic data as the unknowns. More precisely, we fix a compact
real manifold $X$ of dimension $2n$, with a symplectic form $\omega$,
and a pair $(E,H)$ consisting of a set of $C^\infty$ (complex) vector
bundles $E_i$ of ranks $r_i$, and a set of Hermitian metrics $H_i$ on
$E_i$, indexed by the vertices $i\in Q_0$. Let $P_i$ be the frame
$G_i$-bundle of the Hermitian vector bundle $E_i$, where
$G_i=\U(r_i)$, for all $i\in Q_0$, and $\cX_i$ the extended gauge
group of $P_i$ over $(X,\omega)$. Let $P\to X$ be the fibre product of
the principal bundles $P_i\to X$, for all $i\in Q_0$, and $\cX$ the
extended gauge group of $P$ over $(X,\omega)$. Then $P$ is a principal
$G$-bundle, where $G$ is the direct product of the groups $G_i$, for
all $i\in Q_0$, and we have short exact sequences
\[
1\to\cG_i\lto\cX_i\lra{\pr}\cH\to 1,
\qquad
1\to\cG\lto\cX\lra{\pr}\cH\to 1,
\]
where $\cG_i$ is the gauge group of $P_i$, the gauge group of $P$ is
the direct product
\begin{equation}\label{eq:quivers.gauge-group}
\cG=\prod_{i\in Q_0}\cG_i,
\end{equation}
and $\pr\colon\cX\to\cH$ is the fibre product of the group morphisms
$\pr\colon\cX_i\to\cH$, for all $i\in Q_0$.

Let $\cA_i$ is the space of connections on $P_i$. Consider the space
of connections on $P$, denoted
\[
\cA=\prod_{i\in Q_0}\cA_i.
\]
To specify a symplectic structure on $\cA$, we fix a vector
$\alpha\in\RR^{Q_0}_{>0}$, and define an Ad-invariant positive
definite inner product~\eqref{eq:Lie-algebra-inner-product} on the Lie
algebra $\glg$ of $G$, given for all $a,b\in\glg$ by
\begin{equation}\label{eq:quivers.Lie-algebra-inner-product}
(a,b)=-\sum_{i\in Q_0}\alpha_i\Tr(a_i\circ b_i),
\end{equation}
where $a_i,b_i$ are in the Lie algebra $\glg_i$ of $G_i$. Then the
symplectic form~\eqref{eq:SymfC} on $\cA$ becomes
\begin{equation}\label{eq:quivers.SymfC}
\omega_{\cA}(a,b)=-\sum_{i\in Q_0}\alpha_i\int_X \Tr(a_i\wedge b_i)\wedge\frac{\omega^{n-1}}{n-1!},
\end{equation}
for $A\in\cA$, $a,b\in T_A\cA=\Omega^1(\ad E)$. Consider the element
$z$ of the centre of $\glg$ given by $z_i=-\imag c_i\Id_{E_i}$, for
all $i\in Q_0$, for fixed $c_i\in\RR$. By
Proposition~\ref{prop:momentmap-X}, the $\cX$-action on $\cA$ has
equivariant moment map $\mu_\cX\colon \cA\to (\LieX)^*$ given for all
$A\in\cA$, $\zeta\in\LieX$ by
\begin{equation}\label{eq:quivers-mm-connections}
\begin{split}
\langle\mu_\cX(A),\zeta\rangle 
&= \imag\sum_{i\in Q_0}\alpha_i\int_X \Tr\(\xi_i(\imag\Lambda_\omega F_{A_i}-c_i\Id_{E_i})\) \frac{\omega^{n}}{n!}
\\&
+\frac{1}{4}\int_X f\sum_{i\in Q_0}\(\alpha_i\Lambda_\omega^2\Tr F_{A_i}^2 + 4c_i\alpha_i \Tr(\imag\Lambda_\omega F_{A_i})\)\frac{\omega^{n}}{n!},
\end{split}
\end{equation}
where $\xi\defeq A\zeta\in\Lie\cG$
% is the vertical part of $\zeta$ with respect to the connection $A$ 
(so $\xi_i=A_i\zeta\in\LieG_i$), and $\pr(\zeta)=\eta_f$ with $f\in
C_0^\infty(X)$ (see~\eqref{eq:eta_phi}).

Define a Hermitian vector bundle over $X$ by
\[
\cR=\bigoplus_{a\in Q_0}\cR_a,  \text{ with }\cR_a=\Hom(E_{ta},E_{ha}),
\]
where the Hermitian metric is the orthogonal direct sum of the
Hermitian metrics $H_a$ on the vector bundles $\cR_a$, given by the
formulae $(\phi_a,\psi_a)_{H_a}\defeq\Tr(\phi_a\psi_a^{*H_a})$, for
all $\phi_a,\psi_a$ in the same fibre of $\cR_a$. Consider now the
space of $C^\infty$ global sections of $\cR$,
\[
\cS=\bigoplus_{a\in Q_0}\cS_a,    \text{ with }\cS_a=\Gamma(X,\cR_a).
\]
Then $\cS$ has a symplectic form $\omega_\cS$ defined for all
$\phi\in\cS$, $\dot{\phi},\dot{\psi}\in T_\phi\cS\cong\cS$ by
\[
\omega_\cS(\dot{\phi},\dot{\psi})=\imag\sum_{a\in Q_1}\int_X\Tr(\dot{\phi}_a\dot{\psi}_a^*-\dot{\psi_a}\dot{\phi}_a^*)\frac{\omega^{n}}{n!}.
\]
Since $\omega_\cS=d\sigma$, for the 1-form $\sigma$ on $\cS$ given for
all $\phi\in\cS, \dot{\phi}\in T_\phi\cS$ by
\[
\sigma(\dot{\phi})=-\frac{\imag}{2}\sum_{a\in Q_1}\int_X
(\dot{\phi}_a\phi_a^*-\phi_a\dot{\phi}_a^*)\frac{\omega^{n}}{n!},
\]
the canonical $\cX$-action is Hamiltonian, with equivariant moment map
$\mu_\cS\colon\cS\to(\LieX)^*$ given by
$\langle\mu_\cS(\phi),\zeta\rangle=-\sigma(Y_\zeta(\phi))$, where the
infinitesimal action of $\zeta\in\LieX$ on $\cS$ is the vector field
on $\cS$ with value $Y_\zeta(\phi)=\xi\cdot\phi-\pr(\zeta)\lrcorner
d_A\phi$ on $\phi\in\cS$. Here, $\xi=A\zeta\in\LieG$ (so
$\xi_i=A_i(\zeta_i)\in\LieG_i$), the action of $\xi$ on $\phi$ is
given by $(\xi\cdot\phi)_a=\xi_{ha}\phi_a-\phi_a\xi_{ta}$, and
$d_A\phi$ is the covariant derivative with respect to the connection
induced by $A$ on $\cR$. More explicitly,
\begin{equation}\label{eq:quivers-mm-Higgs-fields}
\begin{split}
\langle\mu_\cS(\phi),\zeta\rangle 
&=\imag\sum_{i\in Q_0}\int_X\Tr([\phi,\phi^*]_i\xi_i)\frac{\omega^{n}}{n!}
-\imag\int_X f\sum_{a\in Q_1}\Lambda_\omega d(d_{A_a}\phi_a,\phi_a)_{H_a}\frac{\omega^{n}}{n!}.
\end{split}\end{equation}
% with $\xi_i=A_i(\zeta_i)\in\LieG_i$, for all $\phi\in\cS, \zeta\in\LieX$, with $\xi=A\zeta$ and $f\in C_0^\infty(X)$ such that $df=\pr(\zeta)\lrcorner\omega$.

Fix $\rho\in\RR_{>0}$. Then we consider the space of triples
$\cJ\times\cA\times\cS$, with the symplectic form
\begin{equation}\label{eq:quivers.symplecticT}
\omega_{\alpha,\rho}=\omega_\cJ+4\omega_\cA+4\rho\omega_\cS,
\end{equation}
with $\cJ$ and $\omega_\cJ$ as in Section~\ref{section:cscKmmap}.
% (we omit pullbacks to $\cJ\times\cA\times\cS$).
Adding~\eqref{eq:scmom},~\eqref{eq:quivers-mm-connections}
and~\eqref{eq:quivers-mm-Higgs-fields}, we see that the diagonal
$\cX$-action on $\cJ\times\cA\times\cS$ has equivariant moment map
$\mu_{\alpha,\rho}\colon\cJ\times\cA\times\cS\to(\LieX)^*$ given
by %$\langle\mu_{\alpha,\rho}(J,A,\phi),\zeta\rangle = \langle\mu_\cH(J),\pr(\zeta)\rangle+4\langle\mu_\cX(A),\zeta\rangle +4\rho\langle\mu_\cS(\phi),\zeta\rangle$, or more explicity,
\begin{align}\label{eq:quiver-mm1}&
\langle\mu_{\alpha,\rho}(J,A,\phi),\zeta\rangle=
4\imag\sum_{i\in Q_0}\int_X \Tr\(\xi_i(\alpha_i\imag\Lambda_\omega F_{A_i}+\rho[\phi,\phi^*]_i-\alpha_ic_i\Id_{E_i})\) \frac{\omega^{n}}{n!}
\\\notag &
\! -\!\int_X f\(S_J-\!\sum_{i\in Q_0}\(\alpha_i\Lambda_\omega^2\Tr F_{A_i}^2 + 4c_i\alpha_i\Tr(\imag\Lambda_\omega F_{A_i})\)
+4\rho\!\sum_{a\in Q_1}\imag\Lambda_\omega d(d_{A_a}\phi_a,\phi_a)_{H_a}\)\frac{\omega^{n}}{n!},
\end{align}
for all $(J,A,\phi)\in\cJ\times\cA\times\cS$, $\zeta\in\LieX$, with
$\xi\defeq A\zeta\in\LieG$, $\pr(\zeta)=\eta_f$, $f\in
C_0^\infty(X)$. % $df=\pr(\zeta)\lrcorner\omega$

Consider the $\cX$-invariant subspace
$\cT\subset\cJ\times\cA\times\cS$ of `integrable triples'
$(J,A,\phi)$, defined by the conditions
$J\in\cJ^i,A_i\in(\cA_i)^{1,1}_J, \dbar_{J,A_a}\phi_a=0$, for all
$i\in Q_0, a\in Q_1$ (cf.~\eqref{eq:spaceT}). Since 
\[
\Delta_\omega\lvert\phi_a\rvert_{H_a}^2
=2\imag\Lambda_\omega\dbar\partial\lvert\phi_a\rvert_{H_a}^2
%=2\imag\Lambda_\omega\dbar(\partial_{A_a}\phi_a,\phi_a)_{H_a}
%=2\imag\Lambda_\omega d(\partial_{A_a}\phi_a,\phi_a)_{H_a}
=2\imag\Lambda_\omega d(d_{A_a}\phi_a,\phi_a)_{H_a}
\]
when $\dbar_{A_a}\phi_a=0$, the $\cX$-action %on $\cT$
has equivariant moment map $\mu_{\alpha,\rho}\colon\cT\to(\LieX)^*$
given by
\begin{align}\label{eq:quiver-mm2}
\langle\mu_{\alpha,\rho}&(J,A,\phi),\zeta\rangle=
4\imag\sum_{i\in Q_0}\int_X \Tr\(\xi_i(\alpha_i\imag\Lambda_\omega F_{A_i}+\rho[\phi,\phi^*]_i-\alpha_ic_i\Id_{E_i})\) \frac{\omega^{n}}{n!}
\\\notag&
-\int_X f\(S_J
+2\rho\Delta_\omega\sum_{a\in Q_1}\lvert\phi_a\rvert_{H_a}^2
-\sum_{i\in Q_0}\(\alpha_i\Lambda_\omega^2\Tr F_{A_i}^2 + 4c_i\alpha_i\Tr(\imag\Lambda_\omega F_{A_i})\)\)\frac{\omega^{n}}{n!},
\end{align}
for all $(J,A,\phi)\in\cT$. Defining now $\sigma_i=\alpha_i/\rho$ and
$\tau_i=\alpha_ic_i/\rho$, we see that the vanishing condition
$\mu_{\alpha,\rho}(J,A,\phi)=0$ for a triple $(J,A,\phi)\in\cT$ is
equivalent to the equations
\begin{equation}\label{eq:quivers.Equations.bis.1}
\begin{split}
\sigma_i\imag\Lambda_\omega F_{H_i} + [\phi,\phi^{*H}]_i
%\sum_{a\in h^{-1}(i)}\phi_a\circ\phi_a^{*H_a}-\sum_{a\in t^{-1}(i)}\phi_a^{*H_a}\circ\phi_a
=\tau_i\Id_{E_i},&\\
S_\omega 
+2\rho\Delta_\omega\sum_{a\in Q_1}\lvert\phi_a\rvert_{H_a}^2
-\rho\sum_{i\in Q_0}(\sigma_i\Lambda_\omega^2\Tr F_{H_i}^2
+4\tau_i\Tr(\imag\Lambda_\omega F_{H_i}))
=c',
\end{split}
\end{equation}
expressed in terms of the metrics $\omega$ and $H$, where
$c'\in\RR$. By~\eqref{eq:quivers.Equations.constant}, these equations
are equivalent to the gravitating vortex
equations~\eqref{eq:quivers.Equations.1}, with $c'$ replaced by
another $c\in\RR$.

\subsection{Dimensional reduction}
\label{sub:quivers.DimRed}

We will now consider the invariant solutions of the
K\"ahler--Yang--Mills equations on an equivariant vector bundle over
$M=X\times \GG/P$. Here, $X$ is a compact complex manifold, $\GG$ is a
connected simply connected semisimple complex Lie group, and $P\subset
\GG$ is a parabolic subgroup, so the quotient $\GG/P$ for the $P$-action
by right multiplication on $\GG$ is a flag manifold. The group $\GG$ acts
trivially on the first factor $X$ and in the standard way on
$\GG/P$. The K\"ahler--Yang--Mills equations for the compact complex
manifold $M$, a holomorphic vector bundle $\widetilde{E}\to M$,
and a fixed real parameter $\alpha>0$, are
\begin{subequations}\label{eq:quivers.KYM}
\begin{align}\label{eq:quivers.KYM.a}&
\imag\Lambda_{\widetilde{\omega}}F_{\widetilde{H}}=\mu_{\widetilde{\omega}}(\widetilde{E})\Id_{\widetilde{E}},
\\\label{eq:quivers.KYM.b}&
S_{\widetilde{\omega}}-\alpha\Lambda_{\widetilde{\omega}}^2\Tr F^2_{\widetilde{H}}=C.
\end{align}
\end{subequations}
They involve a pair consisting of a K\"ahler form $\widetilde{\omega}$
on $M$ and a Hermitian metric $\widetilde{H}$ on $\widetilde{E}$, with
normalized slope $\mu_{\widetilde{\omega}}(\widetilde{E})$ defined
by~\eqref{eq:quivers.Slope}.

Let $L\subset P$ be a (reductive) Levi subgroup, and $K\subset \GG$ a
maximal compact Lie subgroup.
% Consider now a $\GG$-equivariant holomorphic vector bundle $F$ over $M$. 
Then the $K$-invariant K\"ahler 2-forms $\omega_\varepsilon$ on the
complex $\GG$-manifold $K/(K\cap P)\cong \GG/P$ are parametrized by
elements $\varepsilon\in\RR_{>0}^\Sigma$ (see~\cite[p.~38,
Lemma~4.8]{AG1}), where $\Sigma$ is a fixed set of `non-parabolic
simple roots', as defined in ~\cite[\S 1.5.1]{AG1}. For a fixed
$\varepsilon\in\RR_{>0}^\Sigma$ and each choice of K\"ahler form
$\omega$ on $X$, we consider the $K$-invariant K\"ahler form on $M$
defined by 
\begin{equation}\label{eq:quivers.omega_epsilon}
\widetilde{\omega}=\omega+\omega_\varepsilon
%\widetilde{\omega}=p^*\omega+q^*\omega_\varepsilon
\end{equation}
(hereafter we omit the symbols for the pullbacks by the canonical
projections $M\to X$, $M\to \GG/P$).
% where $p\colon M\to X$ and $q\colon M\to G/P$ are the canonical projections.

In~\cite{AG1}, the first and the third authors proved that there exist
an infinite quiver $Q$ and a set of relations $\cK$ of $Q$, such that
a $\GG$-equivariant holomorphic vector bundle $\widetilde{E}$ over $M$
is equivalent to a holomorphic $Q$-bundle $(E,\phi)$ over $X$ that
satisfies the relations in $\cK$ (see~\cite[p.~19,
Theorem~2.5]{AG1}). The vertex set $Q_0$ consists of the isomorphism
classes of (finite-dimensional complex) irreducible representations of
$L$. Under this equivalence, the $K$-invariant Hermitian metrics
$\widetilde{H}$ on the vector bundle $\widetilde{E}$ over $M$ are in
bijection with the Hermitian metrics $H$ on the quiver bundle
$(E,\phi)$ over $X$ (see~\cite[\S 4.2.4]{AG1}). Furthermore, for each
choice of K\"ahler form $\omega$ on $X$, a $K$-invariant Hermitian
metric $\widetilde{H}$ satisfies the Hermitian--Yang--Mills
equation~\eqref{eq:quivers.KYM.a} on $\widetilde{E}$ over
$(M,\widetilde{\omega})$ if and only if the corresponding Hermitian
metric $H$ on $(E,\phi)$ over $(X,\omega)$ satisfies the quiver
$(\sigma,\tau)$-vortex equations~\eqref{eq:quivers.Equations.1a}
(see~\cite[\S 4.2.2, Theorem~4.13]{AG1}). Here, the parameters
$\sigma\in\RR_{>0}^{Q_0}$ and $\tau\in\RR^{Q_0}$ are given by
\begin{equation}\label{eq:quivers.Parameters.1}
\sigma_\lambda=\dim_\CC M_\lambda,
\quad
\tau_\lambda=\sigma_\lambda(\mu_{\widetilde{\omega}}(\widetilde{E})-\mu_\varepsilon(\cO_\lambda)), 
\end{equation}
for all $\lambda\in Q_0$, where $M_\lambda$ is an irreducible
representation of $L$ (or $P$) in the isomorphism class $\lambda$,
$\cO_\lambda=\GG\times_P M_\lambda$ is the homogeneous vector bundle
over $\GG/P$ associated to $M_\lambda$, and the normalized slopes
$\mu_\varepsilon(\cO_\lambda) \defeq
\mu_{\omega_\varepsilon}(\cO_\lambda)$, defined
by~\eqref{eq:quivers.Slope}, are explicitly given by~\cite[(4.16), \S
4.2.3]{AG1}. Note that the vortex
equations~\eqref{eq:quivers.Equations.1a}, and the symplectic
interpretation in Section~\ref{sub:quivers.Equations}, make sense for
the infinite quiver $Q$, as $E_\lambda\neq 0$ only for finitely many
$\lambda\in Q_0$, and the quiver $Q$ is locally finite, that is,
$t^{-1}(a)$ and $h^{-1}(a)$ are finite sets for all $a\in Q_1$.

The following correspondence extends these bijections to the
K\"ahler--Yang--Mills equations. It
includes~\cite[Proposition~3.4]{AGG2} for a particular class of
equivariant bundles when $\GG/P=\PP^1$.  As above, $\widetilde{E}$ is
a $\GG$-equivariant holomorphic vector bundle over $M$, and $(E,\phi)$
is the corresponding holomorphic $Q$-bundle over $X$.

\begin{theorem}\label{thm:quivers.Dim-Red}
Let $\omega$ be a K\"ahler form on $X$ and $\widetilde{\omega}$ the
$K$-invariant K\"ahler form on $M$ defined
by~\eqref{eq:quivers.omega_epsilon}.
Let $\widetilde{H}$ be a $K$-invariant Hermitian metric on
$\widetilde{E}$, and $H$ the corresponding Hermitian metric on
$(E,\phi)$.
Then the pair $(\widetilde{\omega},\widetilde{H})$ satisfies the
K\"ahler--Yang--Mills equations~\eqref{eq:quivers.KYM} if and only if
$(\omega,H)$ satisfies the quiver $(\rho,\sigma,\tau)$-vortex
equations~\eqref{eq:quivers.Equations.1}, where $\rho\defeq\alpha$,
and $\sigma\in\RR_{>0}^{Q_0}$ and $\tau\in\RR^{Q_0}$ are given
by~\eqref{eq:quivers.Parameters.1}.
\end{theorem}

\begin{proof}
Let $Q'\subset Q$ be the finite full subquiver with vertex set $Q_0'$
consisting of the vertices $\lambda\in Q_0$ such that $E_\lambda\neq
0$, so $(E,\phi)$ is a $Q'$-bundle over $X$. Let $\widetilde{A}$ and
$A_\lambda$ be the Chern connections of $\widetilde{H}$ and
$H_\lambda$ on the holomorphic vector bundles $\widetilde{E}$ and
$E_\lambda$, respectively, for $\lambda\in Q_0'$. The vector bundles
$E_\lambda$ and the Hermitian metrics $H_\lambda$ on $E_\lambda$, for
$\lambda\in Q_0'$, specify the $K$-action on $\widetilde{E}$ and its
Hermitian metric $\widetilde{H}$, respectively, via the identification
\begin{equation}\label{eq:thm:quivers.Dim-Red.1}
\widetilde{E}=\bigoplus_{\lambda\in Q_0'}E_\lambda\otimes\cO_\lambda
% \widetilde{E}=\bigoplus_{\lambda\in Q_0'}p^*E_\lambda\otimes q^*\cO_\lambda
\end{equation}
between $K$-equivariant $C^\infty$ Hermitian vector bundles, where the
homogeneous vector bundles $\cO_\lambda$ are endowed with their unique
(up to scale) $K$-invariant Hermitian metrics. Furthermore, the Higgs
fields $\phi_a$ and the unitary connections $A_\lambda$, for $a\in
Q_1,\lambda\in Q_0'$, specify the unitary connection $\widetilde{A}$ on
$\widetilde{E}$, given by $d_{\widetilde{A}}=d_{A^\circ}+\theta$, with
$\theta=\beta-\beta^*\in\Omega^1(\ad\widetilde{E})$ and
\begin{equation}\label{eq:thm:quivers.Dim-Red.2}
d_{A^\circ}=\sum_{\lambda\in Q_0'}(d_{A_\lambda}\otimes\Id_{\cO_\lambda}+\Id_{E_\lambda}\otimes d_{A_\lambda'})\circ\pi_\lambda,
%d_{A^\circ}=\sum_{\lambda\in Q_0'}(p^*d_{A_\lambda}\otimes q^*\Id_{\cO_\lambda}+p^*\Id_{E_\lambda}\otimes d_{A_\lambda'})\circ\pi_\lambda,
\quad
\beta=\sum_{a\in Q_1'}\phi_a\otimes\eta_a,
%\beta=\sum_{a\in Q_1'}p^*\phi_a\otimes q^*\eta_a,
\end{equation}
where $A'_\lambda$ is the unique $K$-invariant unitary connection on
$\cO_\lambda$, $\pi_\lambda\colon\widetilde{E}\to E_\lambda\otimes
\cO_\lambda$ are the canonical projections, and $\{\eta_a\mid a\in
t^{-1}(\lambda)\cap h^{-1}(\mu)\}$ is a basis of the space of
$K$-invariant $\Hom(\cO_\lambda,\cO_\mu)$-valued $(0,1)$-forms on
$\GG/P$, for all $\lambda,\mu\in Q_0'$ (see~\cite[\S 3.4.5]{AG1}).

We will use the moment-map interpretations of the
K\"ahler--Yang--Mills equations and the quiver gravitating vortex
equations. Let $\widetilde{\cJ}$, $\widetilde{\cA}$ and $\cX_M$ be the
space of almost complex structures $\widetilde{J}$ on
$(M,\widetilde{\omega})$, the space of unitary connections on
$(\widetilde{E},\widehat{H})$, and the extended gauge group of the
symplectic manifold $(M,\widetilde{\omega})$ and the Hermitian vector
bundle $(\widetilde{E},\widetilde{H})$, respectively.
By~\cite[Proposition~2.1]{AGG}, the $\cX_M$-action on
$\widetilde{\cJ}\times\widetilde{\cA}$, with symplectic form
$\omega_{\widetilde{\cJ}} +4\alpha\omega_{\widetilde{\cA}}$, has
equivariant moment map
$\widetilde{\mu}_{\cX_M}\colon\widetilde{\cJ}\times\widetilde{\cA}
\to(\LieX_M)^*$ given by
\begin{equation}\label{eq:thm:quivers.Dim-Red.3}
\langle\widetilde{\mu}_{\cX_M}(\widetilde{J},\widetilde{A}),\widetilde{\zeta}\rangle
=4\alpha\langle\widetilde{\mu}_{\cG_M}(\widetilde{A}),\widetilde{\xi}\rangle
+\int_M\widetilde{f}S_\alpha(\widetilde{J},\widetilde{A})\frac{\widetilde{\omega}^m}{m!},
%\langle\widetilde{\mu}_{\cH_M}(\widetilde{J},\widetilde{A}),\widetilde{\eta}\rangle,
\end{equation}
for all $\widetilde{\zeta}\in\LieX_M$, where
$\widetilde{\xi}=\widetilde{A}\widetilde{\zeta}\in\LieG_M$,
$\widetilde{f}\in C^\infty_0(M)$ is such that
$d\widetilde{f}=\pr(\widetilde{\zeta})\lrcorner\widetilde{\omega}$,
$m=\dim_\CC M$, and
$\widetilde{\mu}_{\cG_M}\colon\widetilde{\cA}\to(\LieG_M)^*$,
$S_\alpha(\widetilde{J},\widetilde{A})\in C^\infty(M)$
(cf.~\cite[(3.78)]{AGG}) are given by
\begin{subequations}\label{eq:thm:quivers.Dim-Red.4}
\begin{align}&\label{eq:thm:quivers.Dim-Red.4a}
\langle\widetilde{\mu}_{\cG_M}(\widetilde{A}),\widetilde{\xi}\rangle
=\imag\int_M\Tr((\imag\Lambda_{\widetilde{\omega}}F_{\widetilde{A}}-\mu_{\widetilde{\omega}}(\widetilde{E})\Id_{\widetilde{E}})\widetilde{\xi})\frac{\widetilde{\omega}^m}{m!},
\\&\label{eq:thm:quivers.Dim-Red.4b}
S_\alpha(\widetilde{J},\widetilde{A})
=-S_{\widetilde{J}}
+4\alpha\mu_{\widetilde{\omega}}(\widetilde{E})\Tr(\imag\Lambda_{\widetilde{\omega}} F_{\widetilde{A}})
+\alpha\Lambda_{\widetilde{\omega}}^2\Tr F_{\widetilde{A}}^2.
\end{align}
\end{subequations}
By construction,
$\Lambda_{\widetilde{\omega}}F_{\widetilde{A}}+\imag\mu_{\widetilde{\omega}}(\widetilde{E})\Id_{\widetilde{E}}
\in\LieX_M^K$, so $\widetilde{\mu}_{\cG_M}(\widetilde{A})=0$ if and
only if
$\langle\widetilde{\mu}_{\cG_M}(\widetilde{A}),\widetilde{\xi}\rangle=0$
for all $\widetilde{\xi}\in\LieG_M^K$ (where $(-)^K$ means the
fixed-point subspace for the $K$-action). Using the last displayed
formula for $\imag\Lambda_{\widetilde{\omega}}F_{\widetilde{A}}$
in~\cite[\S 4.2.4]{AG1}), we see that
\begin{equation}\label{eq:thm:quivers.Dim-Red.5}
\langle\widetilde{\mu}_{\cG_M}(\widetilde{A}),\widetilde{\xi}\rangle
=\Vol_\varepsilon(\GG/P)\sum_{\lambda\in
  Q_0'}\imag\int_X\Tr\((\sigma_\lambda\imag\Lambda_\omega F_{A_\lambda}+[\phi,\phi^*]_\lambda-\tau_\lambda\Id_{E_\lambda})\xi_\lambda\)
\frac{\omega^n}{n!},
\end{equation}
where $\xi\in\LieG^K$ corresponds to $(\xi_\lambda)_{\lambda\in
  Q_0'}$, with $\xi_\lambda\in\LieG_\lambda$,
by~\cite[Proposition~3.4]{AG1}, $\cG_\lambda$ being the unitary gauge
group of $E_\lambda$, and
$\Vol_\varepsilon(\GG/P)=\int_{\GG/P}\omega_\varepsilon^l/l!$, with
$l=\dim_\CC(\GG/P)$. This gives the correspondence for the vortex
equations~\eqref{eq:quivers.Equations.1a} and the
Hermitian--Yang--Mills equation~\eqref{eq:quivers.KYM.a}. To
compare~\eqref{eq:quivers.Equations.1b} and~\eqref{eq:quivers.KYM.b},
we calculate separately the terms involved
in~\eqref{eq:thm:quivers.Dim-Red.4b}, namely, 
\begin{subequations}\label{eq:thm:quivers.Dim-Red.6}
\begin{align}\label{eq:thm:quivers.Dim-Red.6a}
-S_{\widetilde{J}}&=-S_J+\text{const.},
\\\label{eq:thm:quivers.Dim-Red.6b}
4\alpha\mu_{\widetilde{\omega}}(\widetilde{E})\Tr(\imag\Lambda_{\widetilde{\omega}} F_{\widetilde{A}})
&=4\rho\mu_{\widetilde{\omega}}(\widetilde{E})
\sum_{\lambda\in Q_0'}\sigma_\lambda\Tr(\imag\Lambda_\omega
F_{A_\lambda}\xi_\lambda)
+\text{const.},
\\\notag
\alpha\Lambda_{\widetilde{\omega}}^2\Tr F_{\widetilde{A}}^2
&=\sum_{\lambda\in Q_0'}\(\rho\sigma_\lambda\Lambda_\omega^2\Tr F_{A}^2
-4\rho\sigma_\lambda\mu_\varepsilon(\cO_\lambda)\Tr(\imag\Lambda_\omega
F_{A_\lambda})\)
\\&\label{eq:thm:quivers.Dim-Red.6c}
\!\!\!\!\!\!\!\!\!\!\!\!\!\!\!\!\!\!\!\!\!\!\!\!
-\sum_{a\in Q_1'}4\rho\imag\Lambda_\omega d\Tr(d_{A_a}\phi_a\circ\phi_a^*)
\\&\notag
\!\!\!\!\!\!\!\!\!\!\!\!\!\!\!\!\!\!\!\!\!\!\!\!
-\rho\sum_{a,b\in Q_1'} \Tr(\phi_a\circ\phi_b^*)
\Lambda_{\omega_\varepsilon}^2 d\Tr(\eta_a\wedge d_{A'_b}\eta_b^*+d_{A'_a}\eta_a\wedge\eta_b^*)
+\text{const.},
\end{align}
\end{subequations}
where $A_a$ (resp. $A'_a$) is the connection induced by $A_{ta}$ and
$A_{ha}$ (resp. $A'_{ta}$ and $A'_{ha}$) on the vector bundle
$\Hom(E_{ta},E_{ha})$ (resp. $\Hom(\cO_{ta},\cO_{ha})$), and the sums
in $a,b\in Q_1'$ in~\eqref{eq:thm:quivers.Dim-Red.6b} are constrained
to the condition $ta=tb, ha=hb$ (so that the traces are well
defined). Formula~\eqref{eq:thm:quivers.Dim-Red.6a} follows because
the scalar curvature of $\omega_\varepsilon$ on $\GG/P$ is $K$-invariant
by construction, and hence it is constant, as $K$ acts effectively on
$\GG/P$. Formula~\eqref{eq:thm:quivers.Dim-Red.6b} is obtained taking
traces in the last displayed formula for
$\imag\Lambda_{\widetilde{\omega}}F_{\widetilde{A}}$ in~\cite[\S
4.2.4]{AG1}). We prove~\eqref{eq:thm:quivers.Dim-Red.6c} making the
substitution $F_{\widetilde{A}} =F_{A^\circ} +
d_{A^\circ}\theta+\theta^2$, obtaining
\begin{equation}\label{eq:thm:quivers.Dim-Red.7}
\begin{split}
\Lambda_{\widetilde{\omega}}^2\Tr F_{\widetilde{A}}^2&
=\Lambda_{\widetilde{\omega}}^2\Tr F_{A^\circ}^2
+2\Lambda_{\widetilde{\omega}}^2\Tr(F_{A^\circ}\wedge d_{A^\circ}\theta)
+\Lambda_{\widetilde{\omega}}^2\Tr(\theta^4)
\\&
+2\Lambda_{\widetilde{\omega}}^2\Tr(d_{A^\circ}\theta\wedge\theta^2)
+\Lambda_{\widetilde{\omega}}^2\Tr((d_{A^\circ}\theta)^2)
+2\Lambda_{\widetilde{\omega}}^2\Tr(F_{A^\circ}\wedge\theta^2),
\end{split}
\end{equation}
and calculating the six terms in the right-hand side:
\begin{subequations}\label{eq:thm:quivers.Dim-Red.8}
\begin{align}&\label{eq:thm:quivers.Dim-Red.8a}
\Lambda_{\widetilde{\omega}}^2\Tr F_{A^\circ}^2
=\sum_{\lambda\in Q_0'}\(\sigma_\lambda \Lambda_\omega^2\Tr F_{A_\lambda}^2
-4\sigma_\lambda\mu_\varepsilon(\cO_\lambda)\Tr(\imag\Lambda_\omega F_{A_\lambda})\)
+\text{ const.},
\\&\label{eq:thm:quivers.Dim-Red.8b}
\Lambda_{\widetilde{\omega}}^2\Tr(F_{A^\circ}\wedge d_{A^\circ}\theta)
=0,\,\,%0,\\&\label{eq:thm:quivers.Dim-Red.8c}
\Lambda_{\widetilde{\omega}}^2\Tr(\theta^4)
=0,\,\,%0,\\&\label{eq:thm:quivers.Dim-Red.8d}
\Lambda_{\widetilde{\omega}}^2\Tr(d_{A^\circ}\theta\wedge\theta^2)
=0,
\\&\label{eq:thm:quivers.Dim-Red.8e}
\Lambda_{\widetilde{\omega}}^2\Tr((d_{A^\circ}\theta)^2)
\!=\!\!\sum_{a,b\in Q_1'}\!\(4\Tr\imag\Lambda_\omega(d_{A_a}\phi_a\wedge d_{A_b}\phi_b^*)
-2\Tr(\phi_a\circ\phi_b^*)\Lambda_{\omega_\varepsilon}^2\!\Tr(d_{A_a'}\eta_a\wedge d_{A_b'}\eta_b^*)
\),
\\&\label{eq:thm:quivers.Dim-Red.8f}
\Lambda_{\widetilde{\omega}}^2\Tr(F_{A^\circ}\wedge\theta^2)
\!=\!-2\sum_{a\in Q_1'}\Tr\(\imag\Lambda_\omega
F_{A_{ha}}\circ\phi_a\circ\phi_a^*-\imag\Lambda_\omega F_{A_{ta}}\circ\phi_a^*\circ\phi_a\)
%=-4\sum\Tr\(\imag\Lambda_\omega F_{A_{ha}}\circ(\phi_a\circ\phi_a^*-\phi_a^*\circ\phi_a)\)
%\\&
%\!-\!\sum_{a,b\in Q_1'}\Tr(\phi_a\circ\phi_b^*)\Lambda_{\omega_\varepsilon}^2\!\Tr(F_{A_{ha}'}\wedge [\eta_a,\eta_b^*]),
\\\notag&\hspace*{16.8ex}
+\sum_{a,b\in Q_1'}\Tr(\phi_a\circ\phi_b^*)\Lambda_{\omega_\varepsilon}^2\!
\Tr(F_{A_{ha}'}\wedge\eta_a\wedge\eta_b^*+F_{A_{ta}'}\wedge\eta_a^*\wedge\eta_b).
\end{align}
\end{subequations}
Formula~\eqref{eq:thm:quivers.Dim-Red.8a} follows from the definition
of $d_{A^\circ}$ in~\eqref{eq:thm:quivers.Dim-Red.2}, the identities
$\imag\Lambda_{\omega_\varepsilon}F_{A'_\lambda}
=\mu_\varepsilon(\cO_\lambda)\Id_{\cO_\lambda}$ (see~\cite[Lemma~4.15,
\S 4.2.3]{AG1}), and the fact that $\Lambda_{\omega_\varepsilon}^2\Tr
F_{A'_\lambda}^2\in C^\infty(\GG/P)$ is $K$-invariant, and hence
constant.
The first identity in~\eqref{eq:thm:quivers.Dim-Red.8b} follows by
using~\eqref{eq:thm:quivers.Dim-Red.2} and observing the quiver $Q'$
has no oriented cycles~\cite[Lemma~1.15]{AG1}. The second identity
in~\eqref{eq:thm:quivers.Dim-Red.8b} follows from
$\Tr(\theta\wedge\theta^3) =-\Tr(\theta^3\wedge\theta)$ (as $\theta$
is a 1-form). Using~\eqref{eq:thm:quivers.Dim-Red.2} and the
orthogonal direct-sum decomposition
\begin{equation}\label{eq:thm:quivers.Dim-Red.9}
TM=TX\oplus T(\GG/P)
\end{equation}
(with respect to the metric $\widetilde{\omega}$), it is not difficult
to derive the identity
\begin{equation}\label{eq:thm:quivers.Dim-Red.10}
\Lambda_{\widetilde{\omega}}^2\Tr(d_{A^\circ}\theta\wedge\theta^2)
=-\sum_{a,b,c\in Q_1'}\Tr(\phi_a\circ\phi_b\circ\phi_c^*)\otimes\Lambda_{\omega_\varepsilon}^2
d\Tr(\eta_a\wedge\eta_b\wedge\eta_c^*)+\operatorname{c.c.},
\end{equation}
where ``c.c.'' means complex conjugate. The third identity
in~\eqref{eq:thm:quivers.Dim-Red.8b} now follows because the function
$\Lambda_{\widetilde{\omega}}^2\Tr(d_{A^\circ}\theta\wedge\theta^2)$
is $K$-invariant (by construction), so it equals its average by fibre
integration along the canonical projection $M\to X$, that vanishes
because in~\eqref{eq:thm:quivers.Dim-Red.10},
\[
\int_{\GG/P}\Lambda_{\omega_\varepsilon}^2d\Tr(\eta_a\wedge\eta_b\wedge\eta_c^*)\omega_\varepsilon^l=0.
\]

To prove~\eqref{eq:thm:quivers.Dim-Red.8e}, we
use~\eqref{eq:thm:quivers.Dim-Red.2} with $\theta=\beta-\beta^*$, and
the facts that the quiver $Q'$ has no oriented cycles and the
direct-sum decomposition~\eqref{eq:thm:quivers.Dim-Red.9} is
orthogonal, obtaining
\[
\Lambda_{\widetilde{\omega}}^2\Tr((d_{A^\circ}\theta)^2)
=-2\Lambda_{\widetilde{\omega}}^2\Tr(d_{A^\circ}\beta\wedge d_{A^\circ}\beta^*).
\]
To prove that this equals the right-hand side
of~\eqref{eq:thm:quivers.Dim-Red.8e}, one needs to make another
calculation using~\eqref{eq:thm:quivers.Dim-Red.2}, choosing the basis
$\{\eta_a\}$ as in~\cite[\S 4.2.4]{AG1}, so the pointwise inner
product
\begin{equation}\label{eq:thm:quivers.Dim-Red.11}
-\Tr\imag\Lambda_{\omega_\varepsilon}(\eta_a\wedge\eta_b^*)=\delta_{ab}
\end{equation}
is the Kronecker delta for $ta=tb, ha=hb$, and the orthogonal
decomposition~\eqref{eq:thm:quivers.Dim-Red.9}.

Finally, using~\eqref{eq:thm:quivers.Dim-Red.2} and the decomposition
$\theta=\beta-\beta^*$, one can prove that
\begin{align*}
\Lambda_{\widetilde{\omega}}^2\Tr(F_{A^\circ}\wedge\theta^2)
&=-\Lambda_{\widetilde{\omega}}^2\Tr(F_{A^\circ}\wedge[\beta,\beta^*]),
\end{align*}
where $[\beta,\beta^*]=\beta\wedge\beta^*+\beta^*\wedge\beta$, because
$\Lambda_{\omega_\varepsilon}$ and $\Lambda_{\omega_\varepsilon}^2$
respectively vanish when applied to $(2,0)$ and $(0,2)$-forms, and to
$(1,3)$ and $(3,1)$-forms. To show that this is equal to the
right-hand side of~\eqref{eq:thm:quivers.Dim-Red.8f}, one has to
use~\eqref{eq:thm:quivers.Dim-Red.2} once again,
and~\eqref{eq:thm:quivers.Dim-Red.11}.

Formula~\eqref{eq:thm:quivers.Dim-Red.6c} follows
from~\eqref{eq:thm:quivers.Dim-Red.7},~\eqref{eq:thm:quivers.Dim-Red.8},
and the fact that the connections $A_\lambda$ and $A_\lambda'$ are
unitary, and so putting together the right-hand sides
of~\eqref{eq:thm:quivers.Dim-Red.8e}
and~\eqref{eq:thm:quivers.Dim-Red.8f}, we obtain
\begin{align*}
\Lambda_{\widetilde{\omega}}^2\Tr((d_{A^\circ}\theta)^2)
+2\Lambda_{\widetilde{\omega}}^2\Tr(F_{A^\circ}\wedge\theta^2)
=&-4\sum_{a\in Q_0'}\imag\Lambda_\omega d\Tr(d_{A_a}\phi_a\circ\phi_a^*)
\\&
-\sum_{a,b\in Q_0'}\Tr(\phi_a\circ\phi_b^*)\Lambda_{\omega_\varepsilon}^2
d\Tr(\eta_a\wedge d_{A_b'}\eta_b^*+d_{A_a'}\eta_a\wedge\eta_b^*).
\end{align*}

We can now compare~\eqref{eq:quivers.Equations.1b}
and~\eqref{eq:quivers.KYM.b}. By construction,
$S_\alpha(\widetilde{J},\widetilde{A})\in C^\infty(M)^K$, so
$S_\alpha(\widetilde{J},\widetilde{A})=\text{const.}$ if and only if
the last term in~\eqref{eq:thm:quivers.Dim-Red.4b} vanishes for all
$\widetilde{f}\in C_0^\infty(M)^K$, i.e. $\widetilde{f}=f\circ p_X$
with $f\in C_0^\infty(X)$, $p_X\colon M\to X$ being the canonical
projection. In this case,
\begin{align*}
\int_M\widetilde{f}S_\alpha(\widetilde{J},\widetilde{A})\frac{\widetilde{\omega}^m}{m!}
=&\Vol_\varepsilon(\GG/P)\int_Xf S_{\rho,\sigma,\tau}(J,A,\phi)\frac{\omega^n}{n!},
\end{align*}
where, adding the three identities
in~\eqref{eq:thm:quivers.Dim-Red.6}, we have
\[
S_{\rho,\sigma,\tau}(J,A,\phi)\defeq
-S_J
+\rho\sum_{\lambda\in Q_0'}(\sigma_\lambda \Lambda_\omega^2\Tr
F_{A_\lambda}^2+4\tau_\lambda\Tr(\imag\Lambda_\omega F_{A_\lambda}))
-4\rho\sum_{a\in Q_1'}\imag\Lambda_\omega d\Tr(d_{A_a}\phi_a\circ\phi_a^*).
\]
Combining this and~\eqref{eq:thm:quivers.Dim-Red.5}
in~\eqref{eq:thm:quivers.Dim-Red.3}, we see that
$\langle\widetilde{\mu}_{\cX_M}(\widetilde{J},\widetilde{A}),\widetilde{\zeta}\rangle$
equals~\eqref{eq:quiver-mm1}, up to a factor
$\Vol_\varepsilon(\GG/P)$. This implies the correspondence
for~\eqref{eq:quivers.Equations.1b} and~\eqref{eq:quivers.KYM.b}, as
required.
\end{proof}

Note that the relations in the set $\cK$ have not played any role in
the proof of Theorem~\ref{thm:quivers.Dim-Red}. 

%\subsection{Existence results for the K\"ahler--Yang--Mills equations}
%\label{sub:quivers.Existence-KYM}

\section{Examples}
\label{sec:examples}

\subsection{Solutions in the weak coupling limit}
\label{sub:weakcoupling}

In this section we consider the K\"ahler--Yang--Mills--Higgs equations with coupling constants $\alpha = \beta$. Assuming $\alpha > 0$ and normalizing the first equation in \eqref{eq:KYMH3}, we obtain the system of equations 
\begin{equation}\label{eq:KYMH4}
% \left. \begin{array}{l}
\begin{split}
\Lambda_\omega F_H + \phi^*\hat \mu & =z,\\
S_\omega + \alpha \Delta_{\omega}|\phi^* \hat \mu|^2 + \alpha \Lambda_\omega^2 (F_H\wedge F_H) - 4 \alpha (\Lambda_\omega F_H, z) & =c.
\end{split}
\end{equation}
Following \cite{AGG}, this section is concerned with the existence of solutions of \eqref{eq:KYMH4} in `weak coupling limit' $0<\lvert\alpha\rvert\ll 1$ by deforming solutions $(\omega,H)$ with coupling constants $\alpha = 0$. 

Note that for $\alpha=0$, the coupled equations~\eqref{eq:KYMH4} are the condition that $\omega$ is a constant scalar curvature K\"ahler (cscK) metric on $X$ and $H$ is a solution of the Yang--Mills--Higgs equation, as studied in \cite{MR}. If $\phi=0$, then the existence of solutions of the first equation in \eqref{eq:KYMH4} is equivalent,
by the Theorem of Donaldson, Uhlenbeck and Yau~\cite{D3,UY}, to the polystability of the holomorphic bundle $E^c$ with respect to the K\"ahler class $[\omega] \in H^2(X,\RR)$. For $\phi\neq 0$, Mundet i Riera \cite{MR} gave the following characterization of the existence of solutions.

\begin{theorem}[\cite{MR}]
\label{th:Mundet}
Assume that $\phi\neq 0$ and that $(E^c,\phi)$ is a simple pair. Then, for every fixed K\"ahler form $\omega$, there exists a solution $H$ of the Yang--Mills--Higgs equation if and only if $(E^c,\phi)$ is $z$-stable, in which case the solution is unique.
\end{theorem}

The conditions of simplicity and $z$-stability in the previous theorem are rather technical, and we refer the reader to \cite{MR} for a detailed definition. To give an idea in the language of Section \ref{sub:obstructions}, a sufficient condition for $(E^c,\phi)$ to be a simple pair (see \cite[Definition 2.17]{MR}) is that the Lie algebra $\Lie \Aut(E^c,\phi)$ of infinitesimal automorphisms of $(E^c,\phi)$ vanishes, where $\Lie \Aut(E^c,\phi)$ is given by the Kernel of
$$
\Lie \Aut(X,E^c,\phi) \to H^0(X,TX).
$$
The $z$-stability of the pair $(E^c,\phi)$ can regarded as a version of the geodesic stability in Definition \ref{def:geodstab}, for (weak) geodesic rays $(\omega_t,H_t)$ with $\omega_t = \omega$ constant (see \cite[Definition 2.16]{MR}).

Our next result is a consequence of the implicit function theorem in
Banach spaces, combined with Theorem \ref{th:Mundet} and the moment
map interpretation of the constant scalar curvature K\"ahler metric
equation. The proof follows along the lines of \cite[Theorem
4.18]{AGG}.

\begin{theorem}
\label{thm:DeformationCYMeq3}
Assume that $\phi\neq 0$ and that $(E^c,\phi)$ is a simple
pair. Assume that there is a cscK metric $\omega_0$ on $X$ with
cohomology class $[\omega_0] = \Omega_0$ and that there are no
non-zero Hamiltonian Killing vector fields on $X$. If $(E^c,\phi)$ is
$z$-stable with respect to $\omega_0$, then there exists an open
neighbourhood $U\subset\RR\times H^{1,1}(X,\RR)$ of $(0,\Omega_0)$
such that for all $(\alpha,\Omega)\in U$ there exists a solution of
\eqref{eq:KYMH4} with coupling constant $\alpha$ such that
$[\omega]=\Omega$.
\end{theorem}

We next provide an application of the previous theorem to the
K\"ahler--Yang--Mills equations. Using the notation of
Section~\ref{sub:quivers.DimRed}, we fix a K\"ahler form $\omega_0$ on
$X$, a $K$-invariant K\"ahler form $\omega_\varepsilon$ on $\GG/P$
(with $\varepsilon\in\RR_{>0}^\Sigma$), and the product K\"ahler form
$\widetilde{\omega}_0 =\omega_0+\omega_\varepsilon$ on $M=X\times
\GG/P$. Let $\Omega_0=[\omega_0]$,
$\Omega_\varepsilon=[\omega_\varepsilon]$ and
$\widetilde{\Omega}_0=[\widetilde{\omega}_0]=\Omega_0+\Omega_\varepsilon$
be their cohomology classes on $X$, $\GG/P$ and $M$, respectively. We
also fix a $\GG$-equivariant holomorphic vector bundle $\widetilde{E}$
on $M$, and say $\widetilde{E}$ is \emph{$\GG$-invariantly stable}
(with respect to $\widetilde{\Omega}_0$) if for all $\GG$-invariant
proper subsheaves $\widetilde{E}'\subset\widetilde{E}$, their slopes
with respect to $\widetilde{\Omega}_0$ satisfy
$\mu_{\widetilde{\Omega}_0}(\widetilde{E}') <
\mu_{\widetilde{\Omega}_0}(\widetilde{E})$ (cf.~\cite[Definition~4.6,
\S 4.1.2]{AG1}).

\begin{corollary}\label{thm:quivers.Existence-KYM}
Assume that $\omega_0$ is a constant scalar curvature K\"ahler 
metric on $X$, there are no non-zero Hamiltonian Killing vector fields
on $X$, and $\widetilde{E}$ is $\GG$-invariantly stable
with respect to $\Omega_0$. Then there exists an open neighbourhood
$U\subset\RR\times H^{1,1}(X,\RR)$ of $(0,\Omega_0)$ such that for all
$(\alpha,\Omega)\in U$, there exists a $K$-invariant solution
$(\widetilde{\omega},\widetilde{H})$ of the K\"ahler--Yang--Mills
equations~\eqref{eq:quivers.KYM} on $M$ with coupling constant
$\alpha$ such that $[\widetilde{\omega}]=\Omega+\Omega_\varepsilon$.
\end{corollary}

\begin{proof}
This follows from Theorems~\ref{thm:DeformationCYMeq3}
and~\ref{thm:quivers.Dim-Red}, and the correspondences of~\cite[\S
4]{AG1}. To apply Theorem~\ref{thm:DeformationCYMeq3}, we consider
the holomorphic $Q$-bundle $(E,\phi)$ over $X$ corresponding to
$\widetilde{E}$, and the symplectic form~\eqref{eq:symplecticT} given
by $\omega_\cJ+4\rho\omega_\cA+4\rho\omega_\cS$, i.e., with
$\alpha=\beta$ both equal to $\rho$, where $\omega_\cA$ is now defined
using the invariant inner product~\eqref{eq:quivers.SymfC} with
$\alpha_\lambda=\dim_\CC M_\lambda$.
\end{proof}

Note that this result is not covered by~\cite[Theorem~4.18]{AGG},
since the infinitesimal action by any non-zero element of
$\LieGG/\mathfrak{p}\cong T_P(\GG/P)$ induces a nowhere-vanishing real
holomorphic vector field over $X\times \GG/P$ (where
$\mathfrak{p}\subset\LieGG$ are the Lie algebras of $P\subset \GG$,
respectively).

To illustrate further the scope of application of Theorem
\ref{thm:DeformationCYMeq3}, consider now a compact Riemann surface
$\Sigma$ with genus $g(\Sigma)> 1$, endowed with a K\"ahler metric
$\omega_0$ with constant curvature $-1$. We fix a holomorphic
principal $G^c$-bundle over $\Sigma$ and consider a unitary
representation $\rho \colon G\to\U(W)$, for a hermitian vector space
$W$. We take $F = \mathbb{P}(W)$, endowed with the Fubini--Study
metric, rescaled by a real constant $\tau >0$. Consider the associated
ruled manifold
$$
\mathcal{F} = E^c \times_{G^c} F = \mathbb{P}(E^c \times_{G^c} W).
$$
Denote by $\mathbb{P}(W)^s \subset \mathbb{P}(W)$ the locus of stable points for the linearized $G^c$-action, and set
$$
\mathcal{F}^s = E^c \times_{G^c} \mathbb{P}(W)^s \subset \cF.
$$
Then, if $E^c$ is semistable with respect the K\"ahler class $[\omega_0]$ and $\phi \in H^0(\Sigma,\cF)$ is such that $\phi(\Sigma) \subset \cF^s$, then $(E^c,\phi)$ is $z$-stable, for any $z$ and any value of $\tau$ (see \cite[p. 74]{MR1}). Furthermore, we can also choose $\phi$ such that the pair is simple, by taking its image outside any proper $G^c$-invariant subspace $W' \subset W$.

For the sake of concreteness, consider the case that $G^c=\GL(r,\CC)$
and $\rho$ is the standard representation in $W = \mathbb{C}^r$. Then
$V = E^c \times_{G^c} W$ is a holomorphic vector bundle and $\phi \in
H^0(\Sigma,\cF)$ can be identified with the inclusion
$$
L \subset V
$$
for a holomorphic subbundle $L$ of rank one, as considered by Bradlow and the third author in \cite{BrGP}. Then, the pair $(E^c,\phi)$ is not simple if and only if one can find a holomorphic splitting
$$
V = V' \oplus V''
$$
such that $L$ is contained in $V'$. Identifying $z = -i \lambda \Id \in \mathfrak{u}(r)$ for a real constant $\lambda \in \RR$, the pair $(E^c,\phi)$ is $z$-stable if and only if for any non-zero proper subbundle $V' \subset V$ we have
$$
\frac{\deg(V') + \tau \rk(L \cap V') }{\rk(V')} < \frac{\deg(V) + \tau}{r}.
$$
In this simple situation, the equations \eqref{eq:KYMH4} are for a K\"ahler metric $\omega$ on $\Sigma$ and a hermitian metric $H$ on $V$, and reduce to
\begin{equation*}
% \left. \begin{array}{l}
\begin{split}
i\Lambda_\omega F_H + \tau \pi_L^H & = \lambda \Id,\\
S_\omega + \alpha \tau^2 \Delta_{\omega}|\pi_L^H|^2 & = c',
\end{split}
\end{equation*}
for a suitable real constant $c' \in \RR$, where $\pi_L^H \colon V \to L$ denotes the $H$-orthogonal projection.

\subsection{Gravitating vortices and Yang's Conjecture}
\label{sub:gravvortex}

In this section we apply Theorem \ref{th:Matsushima-type} to find an
obstruction to the \emph{gravitating vortex equations} on the Riemann
sphere, as introduced in \cite{AGG2}. As an application, we give an
alternative affirmative answer to Yang's Conjecture for the
Einstein--Bogomol'nyi equations \cite{Yang3} (see also
\cite[p. 437]{Yangbook}), that shall be compared with the original
proof in \cite[Corollary 4.7]{AGGP}.

Consider $X = \PP^1$, with $G^c = \CC^*$ and $F = \CC$, endowed with the standard hermitian structure. A $\CC^*$-principal bundle on $\PP^1$ is equivalent to a line bundle $\cO_{\mathbb{P}^1}(N)$ of degree $N$, while the Higgs field is $\phi \in H^0(\PP^1,\cO_{\mathbb{P}^1}(N))$. Here we are concerned with the case $\phi \neq 0$, so we assume $N > 0$. Choose a real constant $\tau > 0$, and consider $z = - i \alpha \tau/2$. Then, the K\"ahler--Yang--Mills--Higgs equations \eqref{eq:KYMH3} with coupling constants $\alpha = \beta > 0$ are equivalent to the  gravitating vortex equations \cite{AGG2}
\begin{equation}\label{eq:gravvortexeq1}
\begin{split}
i\Lambda_\omega F_H + \frac{1}{2}(|\phi|_H^2-\tau) & = 0,\\
%\dbar_A \phi & = 0,\\
S_\omega + \alpha(\Delta_\omega + \tau) (|\phi|_H^2 -\tau) & = c,
\end{split}
\end{equation}
where $\omega$ is a K\"ahler metric on $\PP^1$ and $H$ is a hermitian metric on $\cO_{\mathbb{P}^1}(N)$. The constant $c \in \RR$ is topological,
and is explicitly given by
\begin{equation}\label{eq:constantc}
c = 2 - 2\alpha\tau N,
\end{equation}
where have assumed the normalization $\int_{\PP^1}\omega = 2 \pi$.

The first equation in \eqref{eq:gravvortexeq1} is the \emph{abelian vortex equation}. A theorem by Noguchi~\cite{Noguchi}, Bradlow~\cite{Brad} and the third author~\cite{G1,G3} implies that, upon a choice of K\"ahler metric with volume $2\pi$, the equation
$$
i\Lambda_\omega F_H + \frac{1}{2}(|\phi|_H^2-\tau) = 0
$$
admits a (unique) solution provided that $N < \tau/2$. As we will show next, this numerical condition is not enough to ensure the existence of solutions of the coupled system \eqref{eq:gravvortexeq1}.

\begin{theorem}\label{th:Yangconjecture}
If $\phi$ has only one zero, %with multiplicity $N = c_1(L)$
then there are no solutions of the gravitating vortex equations for
$(\PP^1,L,\phi)$.
\end{theorem}

\begin{proof}
Choose homogeneous coordinates $[x_0,x_1]$ on
$\PP^1$
such that $\phi$ is identified with the polynomial
\[
\phi\cong x_0^N.
\]
Here we use the natural identification $H^0(\mathbb{P}^1,L)\cong S^N(\CC^2)^*$, where the right hand side is the space of degree $N$ homogeneous polynomials in the coordinates $x_0,x_1$. By \cite[Lemma 4.3]{AGGP}, it follows that
\[
\Aut(\PP^1,\cO_{\mathbb{P}^1}(N),\phi)\cong\CC^*\rtimes\CC,
\]
which is non-reductive. Consequently, the proof follows from Theorem~\ref{th:Matsushima-type}.
\end{proof}

When the constant $c$ in~\eqref{eq:constantc} is zero, the gravitating vortex equations~\eqref{eq:gravvortexeq1} turn out to be a system of partial differential equations that have been extensively studied in the physics literature, known as the \emph{Einstein--Bogomol'nyi equations}. Based on partial results in \cite{Yang3}, Yang posed a conjecture about non-existence of solutions of the Einstein--Bogomol'nyi equations with $\phi$ having exactely one zero. This conjecture has been recently settled in the affirmative  in \cite{AGGP}. As an application of Theorem \ref{th:Yangconjecture}, we provide here an alternative proof.

\begin{corollary}[{Yang's conjecture}]
\label{cor:YangConjecture}
There is no solution of the Einstein--Bogomol'nyi equations for $\phi$ having exactly one zero.
\end{corollary}

\subsection{Non-abelian vortices on $\mathbb{P}^1$}
\label{sub:nonabelianvortex}

We consider now the case of non-abelian rank-two vortices on the Riemann sphere (corresponding to $G=\U(2)$). 

Let $X=\PP^1$, with $G^c=\GL(2,\CC)$ and $F=\CC^2$, endowed with the
standard hermitian structure. A $G^c$-principal bundle on $\PP^1$ is
equivalent to a split rank-two bundle
\[
V = \cO_{\mathbb{P}^1}(N_1) \oplus \cO_{\mathbb{P}^1}(N_2),
\]
while the Higgs field is 
\[
\phi = (\phi_1,\phi_2) \in H^0(\PP^1,\cO_{\mathbb{P}^1}(N_1)) \oplus H^0(\PP^1,\cO_{\mathbb{P}^1}(N_2)).
\]
We will assume $0 < N_1 \leqslant N_2$. Choose a real constant $\tau > 0$, and consider the central element $z = - i (\alpha \tau/2) \Id$. Then, the K\"ahler--Yang--Mills--Higgs equations \eqref{eq:KYMH3} with coupling constants $\alpha = \beta > 0$ are equivalent to
\begin{equation}\label{eq:gravvortexeqnonab}
\begin{split}
i\Lambda_\omega F_{H} + \frac{1}{2}\phi \otimes \phi^{*_H} & = \frac{\tau}{2} \Id,\\
S_\omega + \alpha(\Delta_\omega + \tau) (|\phi|_{H}^2 - 2\tau) & = c,
\end{split}
\end{equation}
where $\omega$ is a K\"ahler metric on $\PP^1$ and $H$ is a hermitian metric on $V$. The constant $c \in \RR$ is topological,
and is explicitly given by
\begin{equation}\label{eq:constantc2}
c = 2 - 2\alpha\tau (N_1+N_2),
\end{equation}
where have assumed the normalization $\int_{\PP^1}\omega = 2 \pi$.

The first equation in \eqref{eq:gravvortexeqnonab} is the \emph{non-abelian vortex equation}, as studied in \cite{Brad2}. Applying \cite[Theorem 2.1.6]{Brad2} we obtain that this equation admits a solution provided that
\begin{equation}\label{eq:ineq}
N_2 < \frac{\tau}{2} < N_1 + N_2 - \deg([\phi]),
\end{equation}
where $[\phi]$ denotes the line bundle given by the saturation of the image of $\phi \colon \mathcal{O}_{\PP^1} \to V$. We want to show next that condition \eqref{eq:ineq} is not sufficient to solve the full system of equations \eqref{eq:gravvortexeqnonab}. For this, we will apply the Futaki invariant in Proposition \ref{prop:futakibis}. Fix homogeneous coordinates $[x_0,x_1]$ , so that $H^0(\PP^1,\cO_{\mathbb{P}^1}(N_j))\cong S^{N_j}(\CC^2)^*$ is the space of degree $N_j$ homogeneous polynomials in $x_0,x_1$. Following \cite{AGGP}, consider
\begin{equation}\label{eq:phi-Einstein-Bogomonyi}
\phi_j = x_0^{N- \ell_j}x_1^{\ell_j}, \qquad 
\end{equation}
with $0\leq\ell_j<N_j$ (the case $\ell_1 = \ell_2 =0$ corresponds to a Higgs field
$\phi$ that has only one zero). In this case, it can be easily checked that the numerical condition \eqref{eq:ineq} reduces to
\begin{equation}\label{eq:ineq2}
N_2 < \frac{\tau}{2} < N_1 + N_2 - \min\{\ell_1,\ell_2\} - \min\{N_1 - \ell_1,N_2 - \ell_2\},
\end{equation}
and, by choosing suitable values of the parameters $\tau, N_j$, and $\ell_j$, the non-abelian vortex equation admits a solution. To evaluate the Futaki invariant, note that 
the Lie algebra element
\begin{equation}\label{eq:y}
y = \( \begin{array}{cc}
0 & 0 \\
0 & 1
\end{array} \) \in \mathfrak{gl}(2,\CC)
\end{equation}
can be identified with an element in $\Lie\Aut(\PP^1,V,\phi)$ for any choice of $\ell_j$ as before.

\begin{lemma}\label{lem:evaluationfutaki}
\begin{equation}
\langle \cF_{\alpha,\alpha},y\rangle = 2\pi i\alpha(2N_1-\tau)(2\ell_1 - N_1) + 2\pi i\alpha(2N_2-\tau)(2\ell_2 - N_2)
\end{equation}
\end{lemma}

The proof follows along the lines of \cite[Lemma 4.6]{AGGP}, by direct
evaluation of the Futaki invariant using the Fubini--Study metric on
$\mathbb{P}^1$ and the product ansatz $H = H_1 \oplus H_2$, with $H_j$
the Fubini--Study hermitian metric on the line bundle
$\cO_{\mathbb{P}^1}(N_j)$.

As a direct consequence of Proposition \ref{prop:futakibis} and the
previous lemma, we obtain the following.

\begin{theorem}\label{th:Yangconjecture.2}
Let $(V,\phi)$ as before, and assume that \eqref{eq:ineq2} is satisfied. Then, there is no solution of the equations \eqref{eq:gravvortexeqnonab} on $(\PP^1,V,\phi)$, unless the following balancing condition holds
\[
\frac{2\ell_1 - N_1}{2N_2-\tau} + \frac{2\ell_2 - N_2}{2N_1-\tau} = 0.
\]
\end{theorem}

\end{document}